\documentclass[11pt,reqno]{amsart}
\usepackage[colorlinks=true,allcolors=blue,backref=page]{hyperref}
\usepackage{color}
\usepackage{amsmath, amssymb, amsthm}

\usepackage{mathrsfs}
\usepackage{mathtools}
\usepackage[noabbrev,capitalize,nameinlink]{cleveref}
\crefname{equation}{}{}
\usepackage{fullpage}
\usepackage[noadjust]{cite}
\usepackage{graphics}
\usepackage{pifont}
\usepackage{tikz}
\usepackage{bbm}
\usepackage[T1]{fontenc}

\usetikzlibrary{arrows.meta}

\usepackage{environ}
\usepackage{framed}
\usepackage{url}
\usepackage[linesnumbered,ruled,vlined]{algorithm2e}
\usepackage[noend]{algpseudocode}
\usepackage[labelfont=bf]{caption}
\usepackage{cite}
\usepackage{framed}
\usepackage[framemethod=tikz]{mdframed}
\usepackage{appendix}
\usepackage{graphicx}
\usepackage[textsize=tiny]{todonotes}
\usepackage{tcolorbox}
\usepackage{enumerate}
\usepackage[shortlabels]{enumitem}
\allowdisplaybreaks[1]

\newcommand{\E}{\mathbb{E}}

\newcommand{\bx}{\mathbf{x}}
\newcommand{\by}{\mathbf{y}}
\newcommand{\bz}{\mathbf{z}}

\newcommand{\cB}{\mathcal{B}}

\newcommand{\cD}{\mathcal{D}}

\newcommand{\cG}{\mathcal{G}}

\newcommand{\cS}{\mathcal{S}}

\newcommand{\cM}{\mathcal{M}}
\renewcommand{\epsilon}{\varepsilon}

\renewcommand{\P}{\mathbb{P}}

\newcommand{\one}{\mathbbm{1}}

\newtheorem{prop}{Proposition}   
\newtheorem{lemma}[prop]{Lemma} 
\newtheorem{theorem}[prop]{Theorem} 
\newtheorem{claim}[prop]{Claim}

\newtheorem{fact}[prop]{Fact}

\newtheorem{corollary}[prop]{Corollary}

\theoremstyle{definition}

\theoremstyle{remark}
\newtheorem{remark}[prop]{Remark}

\allowdisplaybreaks

\title{The random stable roommates problem typically has no solution} 
\author{Byron Chin}
\address{Department of Mathematics, Massachusetts Institute of Technology}
\email{byronc@mit.edu}

\author{Marcus Michelen}
\address{Department of Mathematics, Northwestern University}
\email{michelen@northwestern.edu}

\begin{document}

\begin{abstract} 
    Assume that $n = 2k$ potential roommates each have an ordered preference of the $n-1$ others. A stable matching is a perfect matching of the $n$ roommates in which no two unmatched people prefer each other to their matched partners. In their seminal 1962 stable marriage paper, Gale and Shapley noted that not every instance of the stable roommates problem admits a stable matching. In the case when the preferences are chosen uniformly at random, Gusfield and Irving predicted in 1989 that there is no stable matching with high probability for large $n$. We prove this conjecture and show that for $n$ sufficiently large, the probability there is a stable matching is at most $n^{-1/17}$.
\end{abstract}

\maketitle

\section{Introduction}

In the stable matching problem, there are an even number $n$ of people whose underlying connections are described by a graph $G$. Each person has a ranked list of preferences over their neighbors in $G$, and the goal is to group the $n$ people into $n/2$ pairs such that no two unmatched people would prefer to pair with each other over their respective matches.  We are interested in whether or not a stable matching exists for a given instance of the graph $G$ and the ranked lists. 

This problem was introduced in the pioneering work of Gale and Shapley \cite{GS:62} who most famously considered the complete bipartite setting of the problem, that is $G = K_{n/2, n/2}$, also known as the stable marriage problem. Under arbitrary preferences, they give an efficient algorithm which terminates in a stable matching thereby showing that a stable matching \emph{always} exists in this bipartite setting. Their algorithm proceeds by an iterative sequence of proposals by one side of the graph, with the other operating under a deferred acceptance protocol. Variants of this highly influential work have been implemented in practice, and for this work, Gale and Shapley were awarded the Nobel Memorial Prize in Economic Sciences.

In this paper we consider the complete graph setting, that is $G = K_n$, also known as the stable roommates problem. Somewhat surprisingly, without the bipartite structure there are instances of preferences for which \emph{no} stable matching exists. This was already noted in Gale and Shapley's original paper \cite{GS:62} in which they  provide an instance with 4 people that does not admit a stable matching. We are interested in the behavior when the preference lists are chosen uniformly at random.

It was conjectured by Gusfield and Irving \cite{GI:89} in 1989 that the probability a stable matching exists vanishes as $n$ grows large. However, despite a fairly detailed understanding of stable matchings under both uniform and arbitrary preferences \cite{P:93, M:15, P:19}, bounds on the probability of existence have proven to be difficult. The only non-trivial  upper bound of $e^{1/2}/2 = 0.8244...$ as $n \to \infty$ was proven by Pittel and Irving \cite{P:94} in 1994. Our main result is the first vanishing upper bound on the probability that a stable matching exists, confirming the conjecture of Gusfield and Irving.
\begin{theorem}\label{th:no-stable-matching}
    Let $X$ be the number of stable matchings on $K_n$ with independent and uniformly random preferences. Then for $n$ sufficiently large we have 
    \[ \P(X \geq 1) \leq n^{-1/17}\,. \]
\end{theorem}

This complements  Pittel's work \cite{P:93} which proved a lower bound of $\P(X \geq 1) \geq \Omega(n^{-1/2})$ by the second moment method, and thus shows that the probability there is a stable matching decays polynomially.  We first discuss prior work on the problem and then outline our approach.

\subsection{Prior Work}
The study of the stable roommates problem initially focused on efficiently finding a stable matching, or deciding that none exist, under arbitrary preferences as posed by Knuth \cite{knuth19761mariages,K:97}.  Knuth in fact conjectured that deciding if an instance of the stable roommates problem has a solution was NP-hard.  This question was negatively resolved by Irving \cite{I:85} in 1985, who extended the iterative algorithm of McVitie and Wilson \cite{MW:71} (which in turn was inspired by the original Gale--Shapley algorithm) to find a stable matching.  This provides a polynomial-time solution to the decision problem of the existence of a stable matching in the stable roommates problem.  Irving also discusses the problem with random preferences.  Based on simulations up to $n = 90$, Irving notes ``The proportion of problem instances of a given size $n$ for which a stable matching exists is clearly a matter of some interest. The computational
evidence suggests that this proportion decreases as $n$ increases, but it is not 
clear whether this proportion tends to a positive limit as $n$ grows large.'' 

Irving's algorithm works with objects he calls ``all-or-nothing cycles'' which restrict the collection of possible partners in a stable matching; his algorithm then either produces a stable matching or a certificate that none exist. 
Seeking to find a ``compact condition that would support the existence
or non-existence of a solution,'' in 1991 Tan \cite{T:91}  introduced the notion of a stable partition and proved that a stable partition always exists under arbitrary preferences for the stable roommates problem.  A stable partition is a generalization of a stable matching which allows for stable cycles rather than just stable pairs.  Tan observes  that the presence of an odd stable cycle is precisely the obstruction to a stable matching. Tan \cite{T:91a} also proved a number of remarkable properties about stable partitions. Notably, for any instance of the stable roommates problem, the collection of odd cycles in every stable partition must be identical. These rigid structural properties motivated both algorithmic ideas and a line of work analyzing the random instance of the problem.  

The question of the stable roommates problem with uniformly random preferences was first raised by Irving \cite{I:85} in 1985, and it was subsequently conjectured by Gusfield and Irving \cite{GI:89} in their 1989 monograph on stable matchings that typically there is no solution.  
Work on this problem has been led by sequences of works of Pittel across the past three decades. As a natural first step, Pittel's 1993 work \cite{P:93}---which may be viewed as a follow-up to an analogous prior work of his on the bipartite case \cite{pittel1992likely}---computes the expectation and variance of the number of stable matchings, finding that $\E[X] = (1+o(1))e^{1/2}$ and $\E[X^2] = \Theta(n^{1/2})$. While this does not imply any upper bound on the probability of a stable matching existing, it proves the only known lower bound that $\P(X \geq 1) = \Omega(n^{-1/2})$. Moreover, his work expressed the probability that any pair of fixed matchings are both stable as an $n$-dimensional integral, and computed the first order asymptotics for this integral; these two tools play a central role in many follow-up works about stable roommates, including ours. 

The first and only non-trivial upper bound was proven by Pittel and Irving \cite{P:94} shortly thereafter by analyzing Tan's algorithm for finding a stable partition. They show that if the algorithm finds a stable matching, then it must find a structure which they call a core configuration, which is closely related to the proposal sequence in the execution of the algorithm. By similar computations, they determine that the expected number of such core configurations is $e^{1/2}/2$, which yields an upper bound on the probability a stable matching exists by the first moment method. 

Following the structural ideas of Tan, an alternative route to ruling out stable matchings is by showing the existence of a stable partition with an odd cycle with high probability. Once again, Pittel \cite{P:19} in 2019 manages to express the relevant probabilities as a high-dimensional integral and compute the first-order asymptotics, showing that the expected number of stable partitions with an odd cycle is $\Theta(n^{1/4})$ and that the second moment is $\Theta(n^{3/4})$. While this does not imply that an odd cycle exists with high probability, the computations imply a number of precise results about nearly stable matchings and maximal stable matchings in the random instance.  In a similar flavor, Mertens \cite{mertens2015small} applied the principle of inclusion and exclusion to obtain an alternating sum for the probability $\P(X\geq 1)$ which is quickly intractable but yields precise values for small $n$.

The efficient algorithms of Irving and Tan have been implemented to simulate uniformly random instances of the stable roommates problem in an effort to support (or refute) the conjecture. The most detailed simulations to date are from work of Mertens \cite{M:15}, who makes the extremely precise conjecture that $\P(X \geq 1) = (1+o(1))e\sqrt{\frac{2}{\pi}} n^{-1/4}$. In making this conjecture, Mertens notes that his simulations support the conjecture that $\P(X \geq 1) \to 0$, which we confirm.  

\subsection{Proof Strategy}
In contrast to all previous works on the problem, we pivot away from analyzing the odd cycle obstruction \cite{T:91, P:19} and the core configurations of Tan's algorithm \cite{P:94} and return to studying the number of stable matchings directly. Our main approach is to show that conditioned on the event that a fixed matching is stable, then in fact \emph{many} matchings are stable with high probability.  In order to execute this strategy, we need to consider the probability that multiple matchings are stable simultaneously. Various works of Pittel \cite{P:93,P:93a,P:94,P:19} prove estimates for the case of one or two matchings being stable by analyzing high-dimensional integrals; while certain aspects of our approach are inspired by Pittel, we analyze various probabilities conditioned on the stability of a fixed matching by working with the conditional probability directly. This leads to both conceptual and technical advantages that we discuss below. We break our discussion of our strategy into two parts, where we first discuss the bird's eye view and then discuss our conditional approach to stability probabilities.

\subsubsection{A bird's eye view of the proof}
Letting $\Pi$ be a matching on $n$ and $\cS$ the set of stable matchings, we want to show that for some parameters $c_1,c_2 > 0$ we have \begin{equation}\label{eq:conditional-X-bound}
    \P(X < n^{c_1} \,|\,\Pi \in \cS) \leq n^{-c_2}
\end{equation}
at which point Markov's inequality---paired with Pittel's bound $\E X = O(1)$---will complete the proof.  To prove \eqref{eq:conditional-X-bound} we will need to understand the landscape of possible stable matchings. 

If we consider another matching $\Pi_1$, the symmetric difference $\Pi \triangle \Pi_1$ consists of cycles of even length of size at least $4$.  To begin, we want to  understand the probability that $\Pi_1$ is stable conditioned on $\Pi$ being stable. We will see that if $\Pi_1 \triangle \Pi$ has exactly $\mu$ cycles then
\begin{equation}\label{eq:two-point-intro}
    \P(\Pi_1 \in \cS \,|\, \Pi \in \cS) \sim 2^\mu n^{-|\Pi_1 \setminus \Pi|} \,.
\end{equation}
A version of this is implicit in Pittel's works \cite{P:93,P:19} if one combines his two point estimate for $\P(\Pi_1 \in \cS \wedge \Pi \in \cS)$ with Stirling's formula, although we reprove this estimate by working directly with the conditional probability (Lemma \ref{lem:first-moment}).  

The simplest sort of matchings to consider are those for which $\Pi_1 \triangle \Pi$ consists of a single cycle.  Letting $\cM_\nu^\circ$ denote the set of $\Pi_1$ so that $\Pi_1 \triangle \Pi$ is a cycle of length $2\nu$, one may sum \eqref{eq:two-point-intro} to see that $\E|\cM_\nu^\circ \cap \cS| \approx \nu^{-1}.$  Setting $\cM_{\leq \nu}^\circ = \cup_{j \leq \nu} \cM_j^\circ$ and summing, we see that $\E|\cM_{\leq n^c}^\circ \cap \cS| = \Theta(\log n)$. In particular, there are a growing number of such single-cycle matchings on average.

Our approach is to first show that $|\cM_{<n^c}^\circ \cap \cS| = \Theta(\log n)$ with probability at least $1 - n^{-\Omega(1)}$. While an application of the second moment method yields this with probability $1-\Omega(1/\log n)$, we perform a more detailed analysis to obtain the stronger high probability guarantee. We prove this in two steps: 
\begin{enumerate}[label=\arabic*.]
\item We will show that with high probability (conditioned on $\Pi \in \cS$), all pairs $\Pi_1,\Pi_2 \in \cM_{\leq n^c}^\circ \cap \cS$ have $(\Pi_1 \setminus \Pi) \cap (\Pi_2 \setminus \Pi) = \emptyset$.  Heuristically, if the overlap consists of $e$ edges and $v$ vertices, we get an $n^{-v}$ fraction of possible pairs whereas the joint probability is $n^e$ times as large. Since the intersection of cycles is a collection of paths, we have $v \geq e+1$, and so only a vanishing fraction of pairs of matchings should overlap. This is formalized using a version of \eqref{eq:two-point-intro} and is the content of Lemma \ref{lem:overlapping-cycles-rare}.  

\item We next show approximate independence: if $\Pi_j \in \cM_{\leq n^c}^\circ$ for $j \leq n^{o(1)}$ are so that $(\Pi_j \triangle \Pi)_{j}$ are disjoint, then \begin{equation}\label{eq:approx-indep-intro}
    \P\left(\bigwedge_j \Pi_j \in \cS \,|\,\Pi \in \cS \right) \sim \prod_{j} \P\left(\Pi_j \in \cS \,|\,\Pi \in \cS \right)\,.
\end{equation}
The intuition from \eqref{eq:two-point-intro} is that conditioned on $\Pi \in \cS$, the events $\{\Pi_j \in \cS\}$ depend mostly  on the symmetric differences $\Pi_j \triangle \Pi$, which are disjoint. We formalize a version of \eqref{eq:approx-indep-intro} in Corollary \ref{cor:k-factor}.  
\end{enumerate}
From here, we prove that $|\cM^\circ_{\leq n^c} \cap \cS|$ behaves roughly like a Poisson random variable by considering a negative exponential moment (see the proof of Lemma \ref{lem:few-cycles-rare}). This shows $|\cM_{\leq n^c}^\circ \cap \cS| = \Theta(\log n)$ with probability at least $1 - n^{-\Omega(1)}$.  

To upgrade to showing that $X = n^{\Omega(1)}$ with high probability, we rely on a structural quasirandomness event on the space of stable matchings. We think of the elements of $\cM^\circ_{\leq n^c} \cap \cS$ as building blocks for other stable matchings. Recall that we showed that cycles corresponding to elements in $\cM^\circ_{\leq n^c} \cap \cS$ are typically disjoint. Thus, for a collection $\Pi_1,\ldots,\Pi_k \in \cM^\circ_{\leq n^c} \cap \cS$, we can combine the cycles and consider the matching $\Pi_{1,2,\ldots,k}$ which has $\Pi_{1,2,\ldots,k} \triangle \Pi = \bigcup_{j = 1}^k (\Pi_j \triangle \Pi)$.  We show that with high probability, $\Pi_{1,2,\ldots,k}$ is \emph{also} stable for every collection $\Pi_1, \ldots, \Pi_k \in  \cM^\circ_{\leq n^c} \cap \cS$.  We show this by proving that the event $\{\Pi_1 \in \cS, \Pi_2 \in \cS, \Pi_{1,2} \not\in \cS\}$ has small probability relative to $\{\Pi_1 \in \cS, \Pi_2 \in \cS\}$, then union bounding over all pairs $\Pi_1,\Pi_2$ with $|\Pi_1 \triangle \Pi|, |\Pi_2 \triangle \Pi| \leq n^c$ (Lemma \ref{lem:combine-cycles-rare}).  This means that we typically have $X \geq 2^{|\cM_{\leq n^c}^\circ \cap \cS|} = n^{\Omega(1)}$, proving \eqref{eq:conditional-X-bound}.

\subsubsection{A conditional approach to stability probabilities}
We now discuss the technical advantages of working with the conditional measure directly. Following the works of Pittel \cite{P:93} and Knuth before him \cite{knuth19761mariages,K:97}, we generate the preference orderings by sampling $X_{i,j} \sim \mathrm{Unif}([0,1])$ independently for all $i \neq j$. Preferences are then defined as follows: $i$ prefers $a$ to $b$ if $X_{i,a} < X_{i,b}$. A matching $\Pi$ is then stable if for every $(i, j) \not\in \Pi$, either $X_{i, \Pi(i)} < X_{i,j}$ or $X_{j, \Pi(j)} < X_{j,i}$. Writing $x_i = X_{i,\Pi(i)}$, the probability that $\Pi$ is stable is computed by integrating the function $\prod_{ij \not\in \Pi} (1-x_ix_j)$ (see \cite{P:93} or Lemma \ref{lem:single:exact}).

To check the stability of another matching $\Pi_1$, we additionally need to check that for each $(i,j) \notin \Pi_1$ we have either $X_{i,\Pi_1(i)}< X_{i,j}$ or $X_{j,\Pi_1(j)}< X_{j,i}$. Set $V$ to be the vertex set of $\Pi \triangle \Pi_1$. Writing $y_i = X_{i, \Pi_1(i)}$ for $i \in V$ and $y_i = x_i$ for $i \in V^c$, Pittel computes the probability that $\Pi$ and $\Pi_1$ are simultaneously stable by integrating the function $\prod_{ij \not\in \Pi \cup \Pi_1} (1-x_ix_j-y_iy_j+(x_i\land y_i)(x_j\land y_j))$ directly (see also Lemma \ref{lem:two-point}). 

Rather than integrate over $\bx$ and $\by$ directly, we take a different approach and condition on the stability of $\Pi$ by dividing by $\prod_{ij \not\in \Pi} (1-x_ix_j)$. Notice that all terms with $(i,j) \in V^c \times V^c$ cancel, and thus we only need to consider $i \in V$. We will show that the contribution from $V \times V$ is negligible when $|V| \leq n^{1/2 - o(1)}$, so we only need to consider $j \in V^c$. Setting $A = \{i \in V: y_i < x_i\}$ and $B = \{i \in V: y_i > x_i\}$, the terms with $i \in A$ simplify to $\one\{y_i < x_i\}$ and the terms with $i \in B$ simplify to $(1-y_ix_j)\one\{y_i > x_i\}$. Thus, the conditional probability is computed by integrating a function like $\prod_{i \in A} \one\{y_i<x_i\} \prod_{ij \in B \times V^c} \frac{1-y_ix_j}{1-x_ix_j}\one\{y_i>x_i\}$. 

Once we condition on $\Pi \in \cS$, then as a guiding heuristic one may think of $x_i$ as behaving approximately like i.i.d.\ exponential random variables of mean $n^{-1/2}$. Thus, the integral of $\one\{y_i < x_i\}$ is roughly the mean of $x_i$ and contributes $n^{-1/2}$ for each $i \in A$. Moreover,
\[ \prod_{ij \in B \times V^c} \frac{1-y_ix_j}{1-x_ix_j} \approx \prod_{ij \in B \times V^c} \exp(-(y_i-x_i)x_j) = \prod_{i \in B} \exp(-(y_i-x_i)\sum_j x_j). \]
Notice now that $\sum_j x_j$ should be well-concentrated around $\sqrt{n}$, and thus $y_i$ behaves like $x_i + \mathrm{Exp}(\sqrt{n})$ and so integrating also yields a factor of $n^{-1/2}$ for each $i \in B$. Thus, we heuristically derive a conditional probability of $n^{-|V|/2} = n^{-|\Pi_1\setminus \Pi|}$, and the $2^\mu$ arises from all possible choices for the sets $A$ and $B$. Notably, our approach does not require any of the precise manipulations and Laplace method estimates performed by Pittel. This allows us to bound and approximate significantly more complex events, such as the probability $k$ matchings are simultaneously stable (see Corollary \ref{cor:k-factor} and Remark \ref{remark:k-point}).

\subsection{Notation}
We use the standard Landau notations $O(\cdot),\Omega(\cdot),\Theta(\cdot),o(\cdot)$ where the asymptotic is as $n \to \infty$.  We often prove matching asymptotic upper and lower bounds and so use the notion $\pm$ to emphasize that the asymptotic holds with each sign.  Finally, we use $\widetilde{O}(\cdot)$ to denote $O(\cdot)$ up to a polylogarithmic term.

\section{Proof of Theorem \ref{th:no-stable-matching}}\label{sec:main}

In this section we state our main lemmas and use them to prove our main result.  For the sake of clarity we postpone the technical work of proving the lemmas to Sections \ref{sec:setup} and \ref{sec:lemmas}, and only introduce the notation relevant to this section here. 

For an even integer $n$, we will consider a fixed matching $\Pi$ and set $\mathcal{S}$ to be the set of stable matchings.  
We are interested in the random variable $X = |\mathcal{S}|$ when $\Pi$ is conditioned to be stable.  Throughout, the probability is taken with respect to uniformly random preference lists.  Pittel's work \cite{P:93} proves an asymptotic for the probability a given matching is stable: 

\begin{lemma}[Pittel, {\cite[Theorem 1]{P:93}}] \label{lem:Pi-stable-computation}
    We have $\P(\Pi \in \cS) = \frac{e^{1/2} }{(n-1)!!}(1\pm O(n^{-1/3}))\,. $
\end{lemma}

Lemma \ref{lem:Pi-stable-computation} also follows from Lemmas \ref{lem:prob-G-UB}, \ref{lem:single:exact} and \ref{lem:integral-compute}. We will use $\Pi$ as shorthand for the event that $\Pi \in \cS$; that is, $\P(\Pi)$ is the probability $\Pi$ is stable and the notation $\P(\, \cdot \, | \, \Pi)$ and $\E[\, \cdot \, | \, \Pi]$ denotes probability and expectation with respect to the measure conditioned on this event. 

For any other matching $\Pi_1$, the symmetric difference $\Pi \triangle \Pi_1$ is a disjoint union of cycles of even length.  We set $\cM^\circ$ to be the set of matchings that differ from $\Pi$ on a single cycle.  Set $\mathcal{M}_\nu^\circ$ to be the subset that differ on a cycle of length exactly $2\nu$ and $\mathcal{M}_{\leq \nu}^\circ = \bigcup_{j \leq \nu}\mathcal{M}_j^\circ $.  
 We will lower bound $X$ by understanding $X_\nu^\circ = |\mathcal{S} \cap \mathcal{M}_\nu^\circ|$ and $X_{\leq \nu}^\circ = |\mathcal{S} \cap \mathcal{M}_{\leq \nu}^\circ| = \sum_{j=2}^\nu X_j^\circ$ instead. 

To begin, we bound the probability that there are two stable matchings differing from $\Pi$ on single cycles that overlap.  We note that  $\Pi_1 \triangle \Pi$ and $\Pi_2 \triangle \Pi$ are edge disjoint if and only if they are vertex disjoint, and so simply write $(\Pi_1 \triangle \Pi) \cap (\Pi_2 \triangle \Pi) = \emptyset$ to indicate that $\Pi_1 \triangle \Pi$ and $\Pi_2 \triangle \Pi$ are disjoint. 
For a parameter $\alpha < 1/4$ we define our first bad event 
\begin{equation}\label{eq:overlapping-cycles-def}
    \cD_1 := \{ \exists~\Pi_1,\Pi_2 \in \cM_{\leq n^{\alpha}}^\circ \cap \cS : (\Pi_1 \triangle \Pi) \cap (\Pi_2 \triangle \Pi) \neq \emptyset \}     
\end{equation}
to be the event that a pair of matchings that overlap are both stable, and claim that it is rare.  
\begin{lemma}\label{lem:overlapping-cycles-rare}
    For $\alpha < 1/4$ we have $
        \P(\cD_1 \,|\, \Pi) \leq n^{-1 + 4\alpha + o(1)}\,.
    $
\end{lemma}

Next, we will show that $X_{\leq \nu}^\circ$ is unlikely to be small.  For parameters $\alpha < 1/4$ and $\gamma > 0$ we define the bad event \begin{equation}\label{eq:few-cycles-def}
    \cD_2 := \left\{X_{\leq n^\alpha}^\circ < \gamma \log n \right\}\,.
\end{equation}
The intuition here is that when the cycles are disjoint, as guaranteed by $\cD_1^c$, they behave roughly independently. This is reflected in the Poisson-like tail given by the second term in the statement below:
\begin{lemma}\label{lem:few-cycles-rare}
    For each $\alpha < 1/4$ and $\gamma,s > 0$ we have \begin{equation*}
        \P(\cD_2 \,|\,\Pi) \leq n^{-1 + 4\alpha + o(1)} + n^{\gamma s-\alpha(1 - e^{-s}) + o(1)} + n^{\gamma s-1/3 + \alpha(1 - e^{-s}) + o(1)}\,.
    \end{equation*}
\end{lemma}

At the moment, we have only considered matchings that differ on a single cycle, and have deduced that there are logarithmically many. In fact, this would already suffice to prove an upper bound of $\Theta(1/\log n)$. We now see how to combine cycles with our next bad event. Here we set $\cM_{\leq n^{\alpha}}$ to be the set of all matchings $\Pi_1$ so that $|\Pi_1 \triangle \Pi| \leq n^{\alpha}$ and define 
\begin{align}
    \cD_3 := \bigg\{\exists~\Pi_1,\Pi_2 \in \cM_{\leq n^{\alpha}}& \cap \cS : (\Pi_1 \triangle \Pi) \cap (\Pi_2 \triangle \Pi) = \emptyset,  (\Pi_{1} \triangle \Pi) \cup (\Pi_2 \triangle \Pi) \text{ has} \leq 100\log n \text{ cycles}\nonumber \\
    &\text{ and } \Pi_{1,2} \not\in \cS \text{ where } \Pi_{1,2} \triangle \Pi = (\Pi_1 \triangle \Pi) \cup (\Pi_2 \triangle \Pi)\bigg\}\,.  \label{eq:combine-cycles-def}
\end{align}
The complement of this event is saying that whenever we have two stable matchings that differ from $\Pi$ on small disjoint sets, we can combine their symmetric differences to obtain a new stable matching. In particular, any disjoint subset of cycles counted in $X_{\leq n^\alpha}^\circ$ combines to form a new stable matching. Once again, we will show that this event is rare. 

\begin{lemma}\label{lem:combine-cycles-rare}
    For $\alpha < 1/4$ we have 
    \begin{equation*}
        \P(\cD_3 \,|\, \Pi) \leq n^{-1 + 4\alpha + o(1)}\,.
    \end{equation*}    
\end{lemma}

Lemmas \ref{lem:overlapping-cycles-rare}, \ref{lem:few-cycles-rare}, and \ref{lem:combine-cycles-rare}---proven in Sections \ref{subsec:overlap-rare}, \ref{subsec:few-rare}, and \ref{subsec:combine-rare}, respectively---arise from a structural understanding of the stable matchings and computing conditional moments of $X_{\leq n^\alpha}^\circ$. For now, we are ready to prove Theorem \ref{th:no-stable-matching}.

\begin{proof}[Proof of Theorem \ref{th:no-stable-matching}]
    Let $\cD = \cD_1 \cup \cD_2 \cup \cD_3$ and note that by Lemmas \ref{lem:overlapping-cycles-rare}, \ref{lem:few-cycles-rare}, and \ref{lem:combine-cycles-rare} we have that $\P(\cD \,|\, \Pi) \leq n^{-1 + 4\alpha + o(1)}  + n^{s \gamma - \alpha(1 - e^{-s}) + o(1)} + n^{-1/3 + s \gamma + \alpha(1 - e^{-s}) + o(1)}$
    for all $\alpha < 1/4$ and $\gamma,s > 0$ to be chosen later.  On $\cD_2^c$ we have that $X_{\leq n^\alpha}^\circ \geq \gamma \log n$ and so on $\cD_1^c$ we have $\gamma \log n$ stable matchings $(\Pi^{(j)})_{j \in [\gamma \log n]}$ so that each of $\Pi^{(j)}\triangle \Pi$ is a single cycle of length at most $n^{\alpha}$ and $(\Pi^{(j)}\triangle \Pi)_{j}$ are mutually  disjoint.  On $\cD_3^c$, we may combine any subset of these to construct a stable matching, implying that on $\cD^c$ we have $X \geq 2^{\gamma \log n} = n^{\gamma \log 2}$.  This shows that \begin{align*}
        \E\left[X \cdot \one\{X < n^{\gamma \log 2}\}\right] & =(n-1)!! \P\left(X < n^{\gamma \log 2}, \Pi \right) = (n-1)!! \P(\Pi)\P\left(X < n^{\gamma \log 2}\,|\, \Pi \right) \\
        &\leq (e^{1/2} + o(1))\P(\cD \,|\, \Pi) \\
        &\leq n^{-1 + 4\alpha + o(1)}  + n^{s \gamma - \alpha(1 - e^{-s}) + o(1)} + n^{-1/3 + s \gamma + \alpha(1 - e^{-s}) + o(1)} \,.
    \end{align*}
    We then use Markov's inequality to bound \begin{align*}
        \P(X \geq 1) &\leq \E\left[X \cdot \one \{X < n^{\gamma \log 2}\} \right] + \P(X \geq n^{\gamma \log 2}) \\
        &\leq n^{-\gamma \log 2 + o(1)} + n^{-1 + 4\alpha + o(1)}  + n^{s \gamma - \alpha(1 - e^{-s}) + o(1)} + n^{-1/3 + s \gamma + \alpha(1 - e^{-s}) + o(1)}  \\
        &\leq n^{-t_\ast  +o(1)}
    \end{align*}
    where we used Lemma \ref{lem:Pi-stable-computation} in the second line and set \begin{equation}
        t_\ast = \max_{\alpha \in [0,1/4],\gamma \geq 0, s \geq 0} \min\left\{\gamma \log 2, 1 - 4\alpha, \alpha(1 - e^{-s}) - s\gamma, \frac{1}{3} - \alpha(1 - e^{-s}) - s\gamma\right\}\,.
    \end{equation}
    We attempt to optimize first by ignoring the constraint $\frac{1}{3}-\alpha(1-e^{-s})-s\gamma$; at the end we will check that this last constraint is larger than the result from optimizing the first three. Since the first and third are increasing and decreasing in $\gamma$ respectively, the optimal choice of $\gamma$ is when the two are equal.  Similarly, the second and third are decreasing and increasing in $\alpha$ respectively, and so the optimal choice of $\alpha$ is when the two are equal.  We set these three expressions to be equal giving the conditions 
    \[ \gamma = \frac{1-e^{-s}}{\log 2 + s}\alpha \quad \text{ and } \quad \alpha = \frac{1}{4+\frac{1-e^{-s}}{\log 2 + s}\log 2}. \]
    Thus, we want to maximize the function $\frac{1-e^{-s}}{\log 2+s}$ over $s \geq 0$. Setting the first derivative equal to 0 gives
    \begin{multline*}
        (\log 2 + s)e^{-s} - (1-e^{-s}) = 0 \iff (s+\log 2+1)e^{-s}=1 \\ \iff (s+\log2+1)e^{-s-\log 2-1}=\frac{1}{2e} \iff s+\log 2 + 1 = -W(-1/2e).
    \end{multline*}
    This suggests setting $s = -W(-1/2e)-1-\log 2 \approx 0.9852$ where $W$ is a branch of the Lambert $W$ function. This choice gives $\alpha \approx 0.2348 < 1/4$ and $\gamma \approx 0.087668 \geq 0$ so $t_\ast \geq 0.060766$. 
    Thus we deduce the bound
    \[ \P(X \geq 1) \leq n^{-0.060766+o(1)} \leq n^{-1/17}. \qedhere \]
\end{proof}

\section{Proofs of main lemmas}\label{sec:setup}
In this section we prove Lemmas \ref{lem:overlapping-cycles-rare}, \ref{lem:few-cycles-rare}, and \ref{lem:combine-cycles-rare}. Many of the proofs involve an understanding of the probability that various collections of matchings are simultaneously stable under a conditional distribution. We state these ingredients in this section and show how they imply our main lemmas, but postpone the true technical work to Section \ref{sec:lemmas}. 

\subsection{Setting things up and a quasirandomness event on preferences}

Recall that we  sample the preferences by generating an array $(X_{i,j})_{i \neq j}$ of i.i.d.\ uniform random variables in $[0,1]$ and we interpret $X_{i,a} < X_{i,b}$ to mean that $i$ prefers $a$ to $b$.  Throughout our computations, we will intersect with a quasirandomness event on the preferences of the conditioned matching.

Letting $x_i = X_{i,\Pi(i)}$, define 
\begin{align}
    \mathcal{G}_\Pi &= \left\{\left|\sum_{i} x_i - \sqrt{n}\right| \leq 10 \log n\right\} \cap \left\{\forall~i : x_i \leq \frac{\log^2 n}{\sqrt{n}}\right\} \nonumber \\
    &\qquad \qquad \cap \left\{ \left|\sum_{(i,j) \in \Pi} x_i x_j - \frac{1}{2} \right| \leq n^{-1/3} \right\} \cap \left\{\left|\sum_{i} x_i^2 - 2\right| \leq n^{-1/3} \right\}\,. \label{eq:G-Pi}
\end{align}

An adaptation of Pittel's argument\footnote{We follow Pittel's proof quite closely and reproduce it in part due to there being a small and ultimately inconsequential typo in Pittel's proof of this fact.  Namely, his version of Lemma \ref{lem:integrand-UB} is \cite[Lemma 4.1]{P:19} and a factor of $1/2$ is dropped when combining his (4.1) and (4.4) into (4.5).  As such, some of our intermediate bounds like Lemma \ref{lem:integrand-UB} differ in small numerical ways.  } will show the following estimate:  
\begin{lemma}\label{lem:prob-G-UB}
    We have $\P(\cG_\Pi^c \,|\,\Pi) \leq e^{-\Omega(\log^2 n)}$.
\end{lemma}

We prove Lemma \ref{lem:prob-G-UB} in Appendix \ref{app:proof-G}. To begin with, we find the asymptotics of the probability that $\Pi_1$ is stable given $\Pi$ is stable (when intersected with $\cG_\Pi$).  The corresponding upper bound in the below is implicitly shown to hold by Pittel \cite{P:93,P:19}; our proof---given in Section \ref{subsec:first-moment}---works directly with the conditional distribution and so is conceptually different and provides upper and lower bounds simultaneously.

\begin{lemma}\label{lem:first-moment}
    Let $\Pi_1$ be so that it differs from $\Pi$ on $\mu$ cycles and $|\Pi \triangle \Pi_1| \leq n^{1/4}$.  Then we have  \begin{equation*}
        \P(\Pi_1,\cG_\Pi\,|\,\Pi) = (1 \pm \widetilde{O}(n^{-1/3})) 2^\mu n^{-|\Pi \setminus  \Pi_1|}.
    \end{equation*}
\end{lemma}

We will combine this with a count of the number of potential matchings $\Pi_1$ that differ from $\Pi$ on a single cycle of fixed size $2\nu$.  
\begin{fact}\label{fact:cycle-count}
     Let $\Pi$ be a matching on $n$ vertices.  Then for $\nu \leq n^{1/4}$ we have \begin{equation}\label{eq:cycle-count-basic}
         |\cM_\nu^\circ| = \binom{n/2}{\nu}\cdot \left(\frac{(\nu-1)!}{2} \right) \cdot 2^{\nu} = (1 \pm O(n^{-1/2})) \frac{n^{\nu}}{2\nu}\,.
     \end{equation}
\end{fact}
\begin{proof}
    We must first choose the $\nu$ pairs $e_1,\ldots,e_\nu$ from $\Pi$ that are involved.   
    In order to choose the cycle determining $\Pi_1$, we view the cycle as contracting each of $e_j$ to some node $v_j$.  Choosing $\Pi_1$ amounts to choosing a cyclic arrangement of the contracted nodes $v_j$ along with multiplying by a factor of $2^{\nu}$ since there are two possible orientations for each edge $e_j$ that contract to a given $v_j$.  In particular, this shows that \begin{equation*} 
        |\cM_\nu^\circ| = \binom{n/2}{\nu}\cdot \left(\frac{(\nu-1)!}{2} \right) \cdot 2^{\nu} = (1 \pm O(\nu^2/n)) \frac{n^{\nu}}{2\nu}\,. \qedhere
    \end{equation*}
\end{proof}

\subsection{Bounding \texorpdfstring{$\cD_1$}{D1}}\label{subsec:overlap-rare}

Aiming to prove $\cD_1$ is rare, we bound the probability two matchings that differ from $\Pi$ on a cycle are simultaneously stable. This bound---proven in Section \ref{subsec:quasirandom}---is particularly effective at handling matchings with intersecting cycles. We will give precise estimates for the complementary case of disjoint cycles later on.

\begin{lemma}\label{lem:no-overlaps-undirected}
    Let $\Pi_1$ and $\Pi_2$ so that $\Pi_i \triangle \Pi$ is a cycle of length $\leq n^{\alpha}$ for $i \in \{1,2\}$ and $\alpha < 1/4$.  Set $\kappa =|\{j:\Pi(j) \neq \Pi_1(j) \neq \Pi_2(j)\}|$.  Then $$\P({\Pi}_1,{\Pi}_2 \,|\,\Pi) \leq (4 + \widetilde{O}(n^{-1/3})) 2^\kappa n^{-|\Pi_1 \cup \Pi_2 \setminus \Pi| }\,.$$
\end{lemma}

To prove Lemma \ref{lem:overlapping-cycles-rare}, we will apply Lemma~\ref{lem:no-overlaps-undirected} along with a union bound so we require a combinatorial lemma counting configurations as well.
\begin{lemma}\label{lem:counting}
    Fix a matching $\Pi$ on $n$ vertices.  For parameters $\nu_1 \leq \nu_2 \leq n^{1/4}$ and $t < \nu_1$, let $F(n,\nu_1,\nu_2,t,s)$ be the number of pairs of cycles  $\Pi_1, \Pi_2$ so that $|{\Pi}_1 \triangle \Pi| = 2\nu_1,  |{\Pi}_2 \triangle \Pi| = 2\nu_2$ and $|({\Pi}_1 \setminus \Pi) \cap ({\Pi}_2 \setminus \Pi)| = t$ and has $s$ components.  Then $$F(n,\nu_1,\nu_2,t,s) \leq  (1 + O(n^{-1/2}))\frac{n^{\nu_1 + \nu_2 - t}}{\nu_1\nu_2} \cdot \left(\frac{t \nu_1 \nu_2}{n} \right)^s \,.$$
\end{lemma}
\begin{proof}
    We recall the proof of Fact \ref{fact:cycle-count} to choose $\Pi_1$: we choose $\nu_1$ edges $e_1,\ldots,e_{\nu_1}$ from $\Pi$, contract them to nodes $v_1,\ldots,v_{\nu_1}$, and identify $\Pi_1$ with a cycle on $v_1,\ldots,v_{\nu_1}$ along with a choice of orientation of the nodes $v_1,\ldots,v_{\nu_1}$.  This gives that the number of choices for $\Pi_1$ is $(1/2 + O(n^{-1/2}))n^{\nu_1}/\nu_1$.
    
    In order to upper bound $F$, we continue to work in this contracted model.  Since $\Pi_2 \setminus \Pi$ shares $t$ edges with ${\Pi}_1 \setminus \Pi$, in the contracted model this means that $\Pi_2$ contains $t$ edges from the cycle on $v_1,\ldots,v_{\nu_1}$ in $s$ connected components.  We claim that the number of ways to choose $t$ edges from a cycle of length $\nu_1$ so that they form exactly $s$ connected components is at most\begin{equation*} \leq
       \nu_1 \binom{t-1}{s-1} \binom{\nu_1 - t - 1}{s-1}\,.
    \end{equation*}
    This is because we have to pick a first element of the cycle, then choose an $s$ composition of $t$ to choose the lengths of the connected components and then an $s$ composition of $\nu_1 -t$ to choose the sizes of the gaps between them.  To complete the count for $\Pi_2$, we may contract each of the $s$ components into nodes $w_1,\ldots,w_s$, choose the remaining $\nu_2 - t - s$ edges from $\Pi$, contract them into nodes $u_1,\ldots,u_{\nu_2 - t - s}$, and choose an unoriented cycle on $\{w_1,\ldots,w_s,u_1,\ldots,u_{\nu_2 - t -s}\}$ giving a count of at most:
     \begin{align*}
        \leq \binom{n/2}{\nu_2 - t - s} 2^{\nu_2 - t - s-1} (\nu_2 - t - 1)! = (1 + O(\nu_2^2 / n)) n^{\nu_2 - t - s} (\nu_2 - t)_{s-1}\,.
    \end{align*}
    In particular, this implies that the number of choices for ${\Pi}_2$ given ${\Pi}_1$ is \begin{align*}
        &\leq (1 + O(\nu_2^2/n))\cdot n^{\nu_2 -t} \cdot \nu_1 \binom{t - 1}{s - 1} \binom{\nu_1 - t - 1}{s - 1} n^{-s} (\nu_2 - t)_{s-1} \\
        &\leq  (1 + O(\nu_2^2/n))\cdot \frac{n^{\nu_2 -t}}{\nu_2} \left(\frac{t \nu_1 \nu_2}{n} \right)^s. \qedhere
    \end{align*}
\end{proof}

\begin{proof}[Proof of Lemma \ref{lem:overlapping-cycles-rare}]
    Notice that $|(\Pi_1 \cup \Pi_2) \setminus \Pi| = \nu_1 + \nu_2 - |(\Pi_1 \setminus \Pi) \cap (\Pi_2 \setminus \Pi)|$. Moreover, for each component of $(\Pi_1 \setminus \Pi) \cap (\Pi_2 \setminus \Pi)$ we have exactly two vertices such that $\Pi(j) \neq \Pi_1(j) \neq \Pi_2(j)$. By Lemmas \ref{lem:no-overlaps-undirected} and \ref{lem:counting}, we have that 
    \begin{align*}
        \P(\cD_1\,|\,\Pi) &\leq \sum_{1 \leq s \leq t \leq \nu_1 \leq \nu_2 \leq n^{\alpha}} (4+\widetilde O(n^{-1/3}))4^sn^{-\nu_1-\nu_2+t}\cdot \frac{n^{\nu_1+\nu_2-t}}{\nu_1\nu_2}\cdot \left(\frac{t\nu_1\nu_2}{n}\right)^s \\
        &\leq (1+o(1)) \sum_{1 \leq t \leq \nu_1 \leq \nu_2 \leq n^\alpha} \frac{16t}{n} \\
        &\leq n^{-1 + 4\alpha + o(1)}. \qedhere
    \end{align*}
\end{proof}

\subsection{Bounding \texorpdfstring{$\cD_2$}{D2}}\label{subsec:few-rare}
To bound $\cD_2$, we will need that the probability that two matchings that differ on disjoint subsets of vertices are both stable is approximately the product that they are individually stable. This is one main advantage of working under the conditional distribution:  we gain approximate independence for most pairs of matchings. 

\begin{lemma}\label{lem:disjoint-factor}
    Let $\Pi_1$ and $\Pi_2$ differ from $\Pi$ on disjoint nodes so that $|\Pi_1 \triangle \Pi|, |\Pi_2 \triangle \Pi| \leq n^{1/4}$.  Then \begin{equation*}
        \P(\Pi_1,\Pi_2, \cG_\Pi \,|\, \Pi) = (1 \pm \widetilde{O}(n^{-1/3}))\P(\Pi_1, \cG_\Pi \,|\, \Pi)\P(\Pi_2, \cG_\Pi \,|\,\Pi)\,.
    \end{equation*}    
\end{lemma}

This is proven in Section \ref{subsec:higher-moments}. As an easy corollary, we have that the probability $k$ cycles are simultaneously stable approximately factors as well: \begin{corollary}\label{cor:k-factor}
    Let $\Pi_1, \ldots, \Pi_k$ be such that $\Pi_i \triangle \Pi$ is a cycle of length at most $n^{1/4}$ for each $i \in [k]$ and so that $(\Pi_i \triangle \Pi)_i$ are mutually  disjoint. Then \begin{equation*}
        \P(\Pi_1, \ldots, \Pi_k, \cG_\Pi \,|\, \Pi) = (1 \pm \widetilde O(kn^{-1/3})) \prod_{i=1}^k \P(\Pi_i, \cG_\Pi \,|\, \Pi)\,.
    \end{equation*}
\end{corollary}

We use Lemma \ref{lem:overlapping-cycles-rare} and Corollary \ref{cor:k-factor} to estimate sufficiently high (factorial) moments of $X_{\leq n^\alpha}^\circ$. In particular, we will restrict to the high probability event that all cycles are disjoint, and use approximate independence to show that the number of cycles behaves similarly to a Poisson random variable. 

First, Lemma \ref{lem:first-moment} combined with Fact \ref{fact:cycle-count} shows that the expected number of cycles is logarithmic in size.  We will in fact need a count for the expected number of tuples.  
\begin{lemma}\label{lem:N_k-expectation}
    For $k \leq 100 \log n$, let $N_k$ be the number of $k$-tuples $(\Pi^{(j)})_{j \in [k]}$ so that $\Pi^{(j)} \triangle \Pi$ is a cycle of length $\leq n^{\alpha}$ for $\alpha < 1/4$ and so that $(\Pi^{(j)} \triangle \Pi)_{j \in [k]}$ are disjoint.  Then $$\E[N_k \one_{\cG_\Pi} \,|\, \Pi] = (1 \pm \widetilde{O}(n^{-1/3}))(h_{n^{\alpha}} - 1)^k $$
    where $h_m = \sum_{j = 1}^m j^{-1}$ is the $m$th harmonic number.
\end{lemma} 
\begin{proof}
    We have \begin{align*}
        \E[N_k \one_{\cG_\Pi} \,|\, \Pi] &= \sum_{(\Pi_j)_{j \in [k]}} \P(\Pi_1,\ldots,\Pi_k, \cG_{\Pi} \,|\, \Pi) = (1 \pm \widetilde{O}(n^{-1/3})) \sum_{(\Pi_j)_{j \in [k]}} \prod_{i = 1}^k \P(\Pi_i, \cG_\Pi \,|\, \Pi) \\
        &= (1 \pm \widetilde{O}(n^{-1/3}))  \sum_{(\Pi_j)_{j \in [k]}} 2^k n^{-\nu_1 - \ldots - \nu_k}
    \end{align*}
    where the second equality is by Corollary \ref{cor:k-factor} and the third is by Lemma \ref{lem:first-moment}, where we set $2\nu_j = |\Pi \triangle \Pi_j|$.  By Fact \ref{fact:cycle-count} we have \begin{equation*}
        \sum_{(\Pi_j)_{j \in [k]}} 2^k n^{-\nu_1 - \ldots - \nu_k} = (1 \pm {O}(n^{-1/2})) \prod_{j = 1}^k \sum_{\nu_j = 2}^{n^{\alpha}} 2 \cdot n^{-\nu_j} \cdot \frac{1}{2}\cdot \frac{n^{\nu_j}}{\nu_j} = (1 \pm {O}(n^{-1/2}))(h_{n^{\alpha}} - 1)^k\,. \qedhere
    \end{equation*}
\end{proof}

Next, we can truncate at moments of logarithmic order as shown by the following lemma.
\begin{lemma}\label{lem:truncation}
    Let $\alpha < 1/4$, then $\P(X_{\leq n^\alpha}^\circ \geq 100 \log n\,|\,\Pi) \leq n^{-1+4\alpha+o(1)}$.
\end{lemma}
\begin{proof}
    Let $\cG = \cG_\Pi \cap \cD_1^c$ and set $k = 100 \log n$ and then note $$\P(X_{\leq n^\alpha}^\circ \geq 100 \log n\,|\,\Pi) \leq \E\left[ \binom{X_{\leq n^\alpha}^\circ}{k} \one_{\cG} \,\big|\, \Pi \right] + \P(\cG^c\,|\,\Pi) = \frac{1}{k!}\E[N_k \one_{\cG} \,|\Pi] + n^{-1 + 4\alpha + o(1)}$$
    by Lemmas \ref{lem:overlapping-cycles-rare} and \ref{lem:prob-G-UB}.  By Lemma \ref{lem:N_k-expectation} we have  $$\frac{1}{k!}\E[N_k \one_{\cG} \,|\Pi] \leq \frac{1}{k!}\E[N_k \one_{\cG_\Pi} \,|\Pi] \leq (1 + o(1)) \frac{(\alpha \log n)^{k}}{k!} \leq n^{-100 + o(1)}$$
    completing the proof.
\end{proof}

With these estimates, we can use a truncated expansion of the moment generating function to prove Lemma \ref{lem:few-cycles-rare}.
\begin{proof}[Proof of Lemma \ref{lem:few-cycles-rare}]
    We intersect with the events $\cG = \cG_\Pi \cap \cD_1^c \cap \{X_{\leq n^\alpha}^\circ \leq 100 \log n\}$ which has total probability at least $1-n^{-1+4\alpha + o(1)}$.  We will take a negative exponential moment; on the event $\cG$ we will have a simpler version of the Taylor expansion to evaluate this exponential moment.
By the binomial theorem, for each $s \geq 0$ and $x \in \mathbb{N}$ we have \begin{equation}
        e^{-sx} = \sum_{k \geq 0} \frac{(-1)^k (1 - e^{-s})^k}{k!} (x)_k\,.
    \end{equation}
and on the event $\cG$ we have that $$e^{-s X_{\leq n^{\alpha}}^\circ} = \sum_{k = 0}^{100 \log n}  \frac{(-1)^k (1 - e^{-s})^k}{k!} N_k =: Y$$
    where $N_k$ is defined in Lemma \ref{lem:N_k-expectation}. We then bound \begin{align*}
        \P(X_{\leq n^\alpha}^\circ \leq \gamma \log n, \cG \,|\, \Pi) &= \P(e^{-sX_{\leq n^\alpha}^\circ} \geq n^{-s\gamma}, \cG \,|\, \Pi) = \P(Y \geq n^{-s\gamma}, \cG \,|\, \Pi)  \leq \P(Y \geq n^{-s \gamma}, \cG_\Pi \,|\,\Pi) \\
        &\leq n^{s\gamma} \E[Y \one_{\cG_\Pi}]
    \end{align*}
    by Markov's inequality. Using Lemma \ref{lem:N_k-expectation} we can bound
    \begin{align*}
        \E[Y\one_{\cG_\Pi}] &= \sum_{k=0}^{100\log n}\frac{(-1)^k(1-e^{-s})^k}{k!}\E[N_k\one_{\cG_\Pi}] \\
        &\leq \sum_{k=0}^{100\log n} \frac{(-1)^k(1-e^{-s})^k}{k!}(h_{n^{\alpha}} - 1)^k + n^{-1/3+o(1)}\sum_{k = 0}^{100 \log n} \frac{(1 - e^{-s})^k}{k!} \left(h_{n^\alpha}\right)^k \\
        &= n^{o(1)}\exp(-{(1-e^{-s})}(h_{n^\alpha} - 1))  + n^{-1/3 + \alpha(1 - e^{-s}) + o(1)} \\
        &= n^{-\alpha(1 - e^{-s}) + o(1)} + n^{-1/3 + \alpha(1 - e^{-s}) + o(1)}\,. \qedhere
    \end{align*}
\end{proof}

\subsection{Bounding \texorpdfstring{$\cD_3$}{D3}}\label{subsec:combine-rare}
Finally, to bound $\cD_3$, we need to understand when two stable matchings that differ on disjoint nodes can be combined to form a new stable matching. This is captured by the following lemma, which shows that the conditional probability the combined matching is not stable is polynomially small.

\begin{lemma}\label{lem:not-union}
     Let $\Pi_1$ and $\Pi_2$ differ from $\Pi$ on disjoint nodes so that $|\Pi_1 \triangle \Pi|, |\Pi_2 \triangle \Pi| \leq n^{\alpha}$ and set $\Pi_{1,2}$ to be the matching so that $\Pi_{1,2} \triangle \Pi =(\Pi_1 \triangle \Pi) \cup (\Pi_2 \triangle \Pi)\,. $  Then \begin{equation}
         \P(\Pi_1,\Pi_2,\neg \Pi_{1,2}, \cG_\Pi\,|\,\Pi) \leq n^{-1+2\alpha+o(1)} \cdot \P(\Pi_1,\cG_\Pi\,|\,\Pi)\P(\Pi_2,\cG_\Pi\,|\,\Pi)
     \end{equation}
\end{lemma}

This is proven in Section \ref{subsec:combine}. It turns out that this failure probability is small enough for us to simply union bound over all potential pairs, at least when $\alpha$ is not too large. 
\begin{proof}[Proof of Lemma \ref{lem:combine-cycles-rare}]
    We will union bound over choices of $\Pi_{1,2}$, which may have at most $200 \log n$ cycles, each of which is length at most $n^\alpha$.  For each choice of $\Pi_{1,2}$ with $k$ cycles, we have at most $2^k$ choices for the pair $(\Pi_1,\Pi_2)$.    To choose $\Pi_{1,2}$ with exactly $k$ cycles, we may choose an ordered tuple $(\Pi^{(j)})$ so that each $\Pi^{(j)} \triangle \Pi$ is a cycle of length at most $n^\alpha$ and the cycles $(\Pi^{(j)} \triangle \Pi)$ are mutually disjoint, and introduce a factor of $1/k!$ to unorder the tuple.  The number of tuples with cycles of length $2\nu_1,\ldots,2\nu_k$ is at most  \begin{equation}
        \prod_{j = 1}^k \frac{n^{\nu_j}}{2 \nu_j} 
    \end{equation}
    by Fact \ref{fact:cycle-count}.  
    By a union bound and using Lemma \ref{lem:not-union}, Corollary \ref{cor:k-factor} and Lemma \ref{lem:first-moment} we have \begin{align*}
        \P(\cD_3\,|\,\Pi) &\leq \sum_{k = 2}^{200 \log n} \frac{1}{k!}\sum_{\nu_1,\ldots,\nu_k \in [2,n^\alpha]} \frac{n^{\sum \nu_j}}{2^k \prod_j \nu_j} \cdot 2^k \cdot n^{-1 + 2\alpha + o(1)} 2^k n^{- \sum \nu_j} \\
        &\leq n^{-1 + 2\alpha + o(1)} \sum_{k = 2}^{200 \log n} \frac{(2\alpha \log n)^k}{k!} \\
        &\leq n^{-1 + 4\alpha + o(1)} \qedhere. 
    \end{align*} 
\end{proof}

\section{A new conditional approach} \label{sec:lemmas}
In this section we dive into the technical work of proving our conditional probability estimates.  Among Pittel's main work from his 1993 paper \cite{P:93} is computing the second (unconditional) moment.  The main engine towards that is an asymptotic for the probability two matchings are stable simultaneously.  Our Lemma \ref{lem:first-moment} is a version of this theorem specialized to the case where the two matchings differ only on a set of size at most $n^{1/4}.$  In order to prove the rest of our bounds on quasirandomness events, we will need estimates that ultimately depend on the stability (or non-stability) of three or four matchings simultaneously.  We take a new approach to studying these probabilities by understanding the conditional measure directly. Through this, we show how our approach gives a conceptually different (and potentially simpler) proof of Pittel's result before proceeding to our other estimates.

\subsection{Setting things up}
We begin with an overview of known results and estimates that form the foundation for our computations.  As before, we   sample the preferences by generating an array $(X_{i,j})_{i \neq j}$ of i.i.d.\ uniform random variables in $[0,1]$ and  interpret $X_{i,a} < X_{i,b}$ to mean that $i$ prefers $a$ to $b$.   We begin with the following basic identity which appears in Pittel \cite{P:93} and is a variant on Knuth's formula in the bipartite case \cite{knuth19761mariages} (see \cite{K:97} for an English translation). 

\begin{lemma}\label{lem:single:exact}
    For a given matching $\Pi$, conditioned on $X_{i,\Pi(i)} = x_i$ we have \begin{equation*}
        \P(\Pi\,|\,X_{i,\Pi(i)} = x_i) = \prod_{(i,j) \notin \Pi} ( 1 - x_i x_j)\,.
    \end{equation*}
\end{lemma}
\begin{proof}
    For each non-edge $(i,j) \notin \Pi$ we just have to check the event that $(i,j)$ do not both prefer each other to their matches.  In particular, we have to check \begin{equation}\label{eq:S-def}
       \{\Pi\} =  \bigwedge_{(i,j) \notin \Pi} \mathcal{S}_{i,j}(\Pi) \qquad \text{ where } \qquad \mathcal{S}_{i,j}(\Pi) = \{X_{i,\Pi(i)} < X_{i,j}\} \cup  \{X_{j,\Pi(j)} < X_{j,i}\}\,.
    \end{equation}
    Notice that conditioned on $\{X_{i,\Pi(i)}\}_{i \in [n]}$ the events $\{\mathcal{S}_{i,j}\}_{(i,j) \notin \Pi}$ are mutually independent with $\P(\mathcal{S}_{i,j}) = 1 - x_i x_j$, completing the proof.
\end{proof}

Integrating the above over choices of $x_i$ yields Lemma \ref{lem:Pi-stable-computation} as computed by Pittel \cite{P:93}. We evaluate the integral in slightly greater generality, which will be crucial to our conditional and two-point estimates. The intuition for the following lemma is that when $\Pi$ is stable, the variables $x_i$ behave roughly as $\mathrm{Exp}(\sqrt{n})$ random variables. Thus, the first moment terms $x_j$ each contribute $1/\sqrt{n}$ and the second moment terms each contribute $2/n$. 

\begin{lemma}\label{lem:integral-compute}
    Let $B_1$ and $B_2$ be disjoint subsets of $[n]$ with $|B_1|, |B_2| \leq n^{1/4}$.  Then \begin{equation*}
        \int_{\bx \in \cG_\Pi} \prod_{j \in B_1} x_j \prod_{j \in B_2} x_j^2 \prod_{(i,j) \notin \Pi} (1 - x_ix_j) \,d\bx = (1 \pm O(n^{-1/3}))\frac{e^{1/2}}{(n-1)!!} \cdot 2^{|B_2|} n^{-|B_1|/2 - |B_2|}\,.
    \end{equation*}
\end{lemma}
Along with the proof of Lemma \ref{lem:prob-G-UB}, this is the only place that we require a direct computation of an integral similar to \cite{P:93}, so for the sake of isolating the ideas from our approach we provide the proof in Appendix \ref{app:proof-G}.

\subsection{Conditional first moment}\label{subsec:first-moment}
In order to compute the probability of a matching being stable conditioned on $\Pi$, we will need to handle the event that multiple matchings are simultaneously stable. We have the following identity yielding an integral representation for the pair probability which also appears in Pittel \cite{P:93}.
\begin{lemma}\label{lem:two-point}
    Let $\Pi_1$ and $\Pi$ be matchings and condition on $X_{i,\Pi(i)} = x_i$ and $X_{i,\Pi_1(i)} = y_i$.  Then $$\P(\Pi, \Pi_1 \,|\, X_{i,\Pi(i)} = x_i, X_{i,\Pi_1(i)} = y_i) = \prod_{ij \notin (\Pi \cup \Pi_1)}(1 - x_ix_j - y_iy_j + (x_i \wedge y_i)(x_j \wedge y_j)) \one_{\bx,\by \in E}$$
    where $E$ is the set with $x_i = y_i$ if $\Pi(i) = \Pi_1(i)$ and for every cycle $\{i_1,\ldots,i_\ell\}$ formed by the symmetric difference $\Pi \triangle \Pi_1$ we have alternating inequalities \begin{equation}\label{eq:orientation-choice}
        x_{i_1} > y_{i_1}, x_{i_2} < y_{i_2},\ldots, x_{i_\ell} < y_{i_\ell} \qquad \text{ or } \qquad x_{i_1} < y_{i_1}, x_{i_2} > y_{i_2},\ldots, x_{i_\ell} > y_{i_\ell}\,.
    \end{equation}
\end{lemma}
\begin{proof}
    Note that first we must have $x_i = y_i$ if $\Pi(i) = \Pi_1(i)$.  By \eqref{eq:S-def} we have 
    \begin{align*}
        \{\Pi,\Pi_1 \} &= \bigwedge_{(i,j) \notin \Pi} \cS_{i,j}(\Pi) \cap \bigwedge_{(i,j) \notin \Pi_1} \cS_{i,j}(\Pi_1) \\
        &= \bigwedge_{(i,j) \in (\Pi \setminus \Pi_1)} \cS_{i,j}(\Pi_1)   \cap   \bigwedge_{(i,j) \in (\Pi_1 \setminus \Pi)} \cS_{i,j}(\Pi) \cap \bigwedge_{(i,j) \notin (\Pi_1 \cup \Pi)}  \cS_{i,j}(\Pi) \cap \cS_{i,j}(\Pi_{1})\,.
    \end{align*}
    Note that for each pair $(i,j) \in \Pi \setminus \Pi_1$ we have \begin{equation} \label{eq:S-cycle-cases}
        \cS_{i,j}(\Pi_1) = \{X_{i,\Pi_1(i)} < X_{i,j}\} \cup \{X_{j,\Pi_1(j)} < X_{j,i}\} = \{y_i < x_i\} \cup \{y_j < x_j\}\,.
    \end{equation}
    Suppose that we have a cycle $i_1,\ldots,i_\ell$ formed by $\Pi \triangle \Pi_1$ and assume without loss of generality that $(i_1,i_2) \in \Pi \setminus \Pi_1$.  
    If we assume that the first inequality in \eqref{eq:S-cycle-cases} occurs for $i = i_1$ and $j = i_2$, we only have one choice for which inequality holds for $\cS_{\Pi}(i_2,i_3), \cS_{\Pi_1}(i_3,i_4)$ and so on.  This shows that one must have $\bx,\by \in E$.  To handle the pairs $(i,j) \notin (\Pi_1 \cup \Pi)$, note that conditioned on $\{X_{i,\Pi(i)},X_{i,\Pi_1(i)}\}_{i \in [n]}$, the events $\{\cS_{i,j}(\Pi) \cap \cS_{i,j}(\Pi_1)\}_{(i,j) \notin (\Pi \cup \Pi_1)}$ are mutually independent.  By the principle of inclusion-exclusion we have \begin{equation*}\P(\cS_{i,j}(\Pi) \cap \cS_{i,j}(\Pi_1)) = 1 - x_ix_j - y_iy_j + (x_i \wedge y_i) (x_j \wedge y_j)\,. \qedhere \end{equation*} 
\end{proof} 

We will often be working in the setting where we have a matching $\Pi$ whose stability we are conditioning on.  With this in mind, it will be useful to have notation for the choice in \eqref{eq:orientation-choice}.  We think of this as an orientation of the cycles in $\Pi_1 \triangle \Pi$, and may be defined by fixing the sets $A = \{i : X_{i,\Pi_1(i)} < X_{i,\Pi(i)}\}$ and $B = \{i : X_{i,\Pi_1(i)} > X_{i,\Pi(i)}\}.$ When considering the event $\Pi_1$ and working conditioned on $\Pi$, we will write $\vec{\Pi}_1$ to indicate that we have fixed choices for the sets $A$ and $B$. 

We will use the law of total expectation, and so the main work for proving Lemma~\ref{lem:first-moment} is to identify the probability of a given orientation $\vec{\Pi}_1$ holding conditioned on $\Pi$ and $\bx$.  As a key first step, we first show that the contribution from having some $j$ with $X_{j,\Pi_1(j)} \geq 2 \log^2 n / \sqrt{n}$ is negligible.  We will later need a similar statement for when we seek stability of two matchings $\Pi_1$ and $\Pi_2$ simultaneously with $\Pi$ and so we prove our lemma in this setting.  

\begin{lemma}\label{lem:no-big-y}
    Let $\Pi_1$ and $\Pi_2$ be matchings with $|\Pi \triangle \Pi_1| \leq n^\alpha$ and $|\Pi \triangle \Pi_2| \leq n^\alpha$ for $\alpha < 1/2$ so that $\Pi \triangle \Pi_1$ and $\Pi \triangle \Pi_2$ are disjoint.  Let $\cB = \{\exists~j \in [n], k \in [2]: X_{j,\Pi_k(j)} \geq 2 \log^2 n /\sqrt{n}\}$.  Then for $\bx \in \cG_\Pi$ we have $$\P(\vec{\Pi}_1,\vec{\Pi}_2, \cB \,|\,\Pi,\bx) \leq e^{-\Omega(\log^2 n)} \prod_{i \in A_1 \cup A_2} \frac{x_i}{\sqrt{n}}$$
    where here $A_{k} = \{j \in [n] : X_{j,\Pi_k(j)} < X_{j,\Pi(j)}\}\,.$
\end{lemma}
\begin{proof}
    Set $V = A_1 \cup A_2 \cup B_1 \cup B_2$ where $B_k = \{j \in [n]: X_{j, \Pi_k(j)} > X_{j, \Pi(j)}\}$. Define the vectors $\by$ and $\bz$ via $y_j = X_{j,\Pi_1(j)}$ for $j \in A_1 \cup B_1$ and $z_j = X_{j,\Pi_2(j)}$ for $j \in A_2 \cup B_2$. 
    Set $\vec{E}$ to be the set of $\bx,\by,\bz$ which corresponds to orientations $\vec{\Pi}_1,\vec{\Pi}_2$ and note
        \begin{align}
        \{\vec{\Pi}_1,\vec{\Pi}_2 \} &\subseteq \left\{\vec{E} \wedge \bigwedge_{(i,j) \in V \times V^c} \cS_{i,j}(\Pi_1) \cap \cS_{i,j}(\Pi_2) \right\} \label{eq:event-triple-UB}
    \end{align}
    which we note holds for arbitrary $\vec{\Pi}_1$ and $\vec{\Pi}_2$.  Conditional on $\bx,\by,\bz$, the events are mutually independent across pairs $(i,j) \in V \times V^c$ and so on $\vec{E}$ and $\cG_\Pi$ we may bound \begin{align*}
        \P(\vec{\Pi}_1,\vec{\Pi}_2, \cB \,|\,\Pi,\bx,\by,\bz) \leq \int_{[0,1]^V \cap \cB} \prod_{(i,j) \in V \times V^c} \frac{1 - x_ix_j - y_iy_j + (x_i \wedge y_i) (x_j \wedge y_j)}{1 - x_i x_j}\,d\by
    \end{align*}
    where in the integral we relabeled $(z_j)_{j \in A_2 \cup B_2}$ by $(y_j)_{j \in A_2 \cup B_2}$ using the assumption that $\Pi \triangle \Pi_1$ and $\Pi \triangle \Pi_2$ are disjoint.  Setting $A = A_1 \cup A_2$ and $B = B_1 \cup B_2$ we note that
        \begin{align*}
        \int_{[0,1]^V \cap \cB} \prod_{ij \in V \times V^c} \frac{1 - x_ix_j - y_iy_j + (x_i \wedge y_i) (x_j \wedge y_j)}{1 - x_i x_j}\,d\by &= \int_{[0,1]^V \cap \cB} \prod_{i \in A}\one_{y_i < x_i} \prod_{ij \in B \times V^c} \frac{1 - y_ix_j}{1 - x_i x_j}\,d\by \\
        &\leq \prod_{i \in A} x_i \int_{[0,1]^B}\one_{\cB}\prod_{(i,j) \in B \times V^c}\frac{1 - y_ix_j}{1 - x_i x_j}\,d\by\,.
    \end{align*}
    We now bound the latter integral by replacing $\one_{\cB} \leq \sum_{b \in B} \one\{y_b \geq 2 \log^2 n /\sqrt{n}\} =: \sum_{b \in B}\one_{\cB_b}$ and noting that for a fixed $b \in B$ we have \begin{align*}
        \int_{[0,1]^B}\one_{\cB_b}&\prod_{(i,j) \in B \times V^c}\frac{1 - y_ix_j}{1 - x_i x_j}\,d\by \\
        &\leq (1 + \widetilde{O}(n^{-1 + \alpha})) \int_{[0,1]^B} \one_{\cB_b} \exp\left(- \sum_{(i,j) \in B \times V^c} (y_i - x_i)x_j \right)\,d\by \\
        &=(1 + \widetilde{O}(n^{-1 + \alpha})) \int_{[0,1]^B} \one_{\cB_b} \exp\left(- \sum_{i \in B} (y_i - x_i) \sqrt{n}(1 + \widetilde{O}(n^{-1/2})) \right)\,d\by \\
        &\leq e^{-\Omega(\log^2 n)} n^{-|B|/2}\,.
    \end{align*}
    Summing over all $b \in B$ and noting that $|B| = |A|$ completes the proof. 
\end{proof}

We are now ready to understand the conditional probability  $\P(\vec{\Pi}_1 \,|\,\Pi,\bx).$ 

\begin{lemma}\label{lem:two-point-conditional}
        Let $\Pi_1$ be a matching with $|\Pi \triangle \Pi_1| \leq n^\alpha$ for $\alpha < 1/2$. For $\bx \in \cG_\Pi$ we have $$\P(\vec{\Pi}_1 \,|\,\Pi,\bx) = (1 \pm \widetilde{O}(n^{-1 + 2\alpha})) \prod_{i \in A} \frac{x_i}{\sqrt{n}} \,.$$
\end{lemma}
\begin{proof}
    Set $V = A \cup B$ and note that by Lemma \ref{lem:two-point} we have \begin{align*}
        \P(\vec{\Pi}_1\,|\,\Pi, \bx) &= \prod_{ij \not\in \Pi} (1-x_ix_j)^{-1} \int_{[0,1]^{V}} \prod_{(i,j) \not\in \Pi \cup \Pi_1} (1-x_ix_j-y_iy_j+(x_i\land y_i)(x_j \land y_j)) \, d\by \\
        &= (1 \pm \widetilde{O}(n^{-1 + \alpha})) \int_{[0,1]^{V}} \prod_{(i,j) \not\in \Pi \cup \Pi_1} \left(\frac{1-x_ix_j-y_iy_j+(x_i\land y_i)(x_j \land y_j)}{1 - x_ix_j}\right) \, d\by
    \end{align*}
    where in the second line we used that $|\Pi \triangle \Pi_1| \leq n^\alpha$ and $|x_i x_j| \leq \widetilde{O}(n^{-1})$ on $\cG_{\Pi}$.  By Lemma \ref{lem:no-big-y} the contribution from the set $\cB = \cup_{j \in B}\{y_j \geq \frac{2\log^2 n}{\sqrt{n}}\}$ is negligible. This shows that we may intersect with $\cB^c$; on $\cB^c$ we have  $\left(\frac{1-x_ix_j-y_iy_j+(x_i\land y_i)(x_j \land y_j)}{1 - x_ix_j}\right) = 1 - \widetilde{O}(n^{-1})$ for all pairs $ij$.  As such, we have \begin{align*}
        \prod_{(i,j) \not\in \Pi \cup \Pi_1}& \left(\frac{1-x_ix_j-y_iy_j+(x_i\land y_i)(x_j \land y_j)}{1 - x_ix_j}\right) \\
        &= (1 \pm \widetilde{O}(n^{-1 + 2\alpha}))\prod_{(i,j) \in V \times V^c} \left(\frac{1-x_ix_j-y_iy_j+(x_i\land y_i)(x_j \land y_j)}{1 - x_ix_j}\right) \\
        &= (1 \pm \widetilde{O}(n^{-1 + 2\alpha}))\prod_{(i,j) \in B \times V^c}\left(\frac{1 - y_i x_j}{1 - x_ix_j}\right)
    \end{align*}
    where the second equality is by noting that for $i \in A$ the term in the product is $1$.  On the event $\cB^c$ and $\cG_\Pi$ we have $$\prod_{(i,j) \in B \times V^c}\left(\frac{1 - y_i x_j}{1 - x_ix_j}\right) = (1 \pm \widetilde{O}(n^{-1 + \alpha})) \exp\left(-\sum_{(i,j) \in B \times V^c} (y_i - x_i) x_j \right)\,.$$
    Using that on $\cG_\Pi$ we have $\sum_{j \in V^c} x_j = \sqrt{n}(1 \pm \widetilde{O}(n^{-1+\alpha}))$ we have  \begin{align*}
        \P(\vec{\Pi}_1,\cB^c \,|\, \Pi, \bx) &= (1 \pm \widetilde{O}(n^{-1+2\alpha}) \prod_{i \in A} x_i \int_{[0,1]^B, \cB^c}\exp\left(-\sum_{(i,j) \in B \times V^c} (y_i - x_i) x_j \right)\,d\by \\
        &= (1 \pm \widetilde{O}(n^{-1 + 2\alpha})) \prod_{j \in A} x_j \prod_{j \in B} n^{-1/2}\,. \qedhere
    \end{align*}
\end{proof}

The proof of Lemma \ref{lem:first-moment} will now follow easily:

\begin{proof}[Proof of Lemma \ref{lem:first-moment}]
    It is sufficient to prove that for a fixed orientation $\vec{\Pi}_1$ we have $$\P(\vec{\Pi}_1,\cG_\Pi \,|\,\Pi) = (1 \pm \widetilde{O}(n^{-1/3})) n^{-|\Pi \setminus \Pi_1| }\,.$$
    With this in mind, fix sets $A$ and $B$ providing the oriented version of $\vec{\Pi}_1$.  Letting $\bx = \{X_{i,\Pi(i)}\}_{i \in [n]}$, by the law of total expectation we may write  \begin{align*}
        \P(\vec{\Pi}_1,\cG_\Pi\,|\,\Pi)\P(\Pi) = \P(\vec{\Pi}_1,\cG_\Pi,\Pi) = \E_{\bx}[\P(\vec{\Pi}_1,\cG_\Pi,\Pi \,|\,\bx)  ] = \E_\bx[\one_{\cG_\Pi}\P(\vec{\Pi}_1\,|\Pi,\bx)\cdot \P(\Pi\,|\,\bx)]\,.
    \end{align*}
    By Lemmas \ref{lem:single:exact}, \ref{lem:integral-compute} and  \ref{lem:two-point-conditional} we have 
    \begin{align*}
        \E_\bx[\one_{\cG_\Pi}\P(\vec{\Pi}_1\,|\Pi,\bx)\cdot \P(\Pi\,|\,\bx)] &=  (1 \pm \widetilde{O}(n^{-1/2})) n^{-|B|/2} \E_{\bx}[\one_{\cG_\Pi} \prod_{j \in A} x_j \prod_{ij \notin \Pi}(1 - x_i x_j)  ] \\
        &= (1 \pm \widetilde{O}(n^{-1/3})) \frac{e^{1/2}}{(n-1)!!} n^{-|\Pi \setminus \Pi_1|}\,. \qedhere
    \end{align*}
\end{proof}

\subsection{Conditional higher moments}\label{subsec:higher-moments}
The same approach will allow us to show that the probability that a pair of matchings are both stable conditioned on $\Pi$ approximately factors when they differ on different nodes. This will allow us to compute conditional pair probabilities and higher moments.

\begin{proof}[Proof of Lemma~\ref{lem:disjoint-factor}]
    It is sufficient to prove the analogous estimate for the directed versions $\vec{\Pi}_1,\vec{\Pi}_2$.  Let $A_j, B_j$ be the corresponding $A$ and $B$ sets for $\vec{\Pi}_j$ for $j \in \{1,2\}$ and set $V_j = A_j \cup B_j$ and $V = V_1 \cup V_2$.  As before, we will condition on the vectors $\bx,\by,\bz$ with $x_i = X_{i,\Pi(i)}, y_i = X_{i,\Pi_1(i)}$ and $z_i = X_{i,\Pi_2(i)}$.  Letting $\vec{E}$ denote the event on $\bx,\by,\bz$ corresponding to the orientations $\vec{\Pi}_1,\vec{\Pi}_2$, we first claim that on $\cG_\Pi, \vec{E}$ we have 
    \begin{equation}\label{eq:double-disjoint-factor}
        \P(\vec{\Pi}_1,\vec{\Pi}_2 \,|\, \Pi, \bx,\by,\bz) = (1 \pm \widetilde{O}(n^{-1/2})) \prod_{(i,j) \in B_1 \times V^c} \left(\frac{1 - y_ix_j}{1 - x_i x_j} \right)\prod_{(i,j) \in B_2 \times V^c} \left(\frac{1 - z_ix_j}{1 - x_i x_j} \right)\,.
    \end{equation}
    To see this, first note that \begin{equation}\label{eq:stability-event-triple}
        \{\Pi_1, \Pi_2, \Pi\} = \bigwedge_{(i,j) \notin \Pi} \cS_{i,j}(\Pi) \cap \bigwedge_{(i,j) \notin \Pi_1}\cS_{i,j}(\Pi_1) \cap \bigwedge_{(i,j) \notin \Pi_2} \cS_{i,j}(\Pi_2)\,.
    \end{equation}
    For $(i,j) \in V^c \times V^c$ we have $\cS_{i,j}(\Pi) = \cS_{i,j}(\Pi_1) = \cS_{i,j}(\Pi_2)$.  For $(i,j) \in A \times V^c$ we have that each of $\cS_{i,j}(\Pi_1)$ and $\cS_{i,j}(\Pi_2)$ hold conditioned on $\cS_{i,j}(\Pi)$.  This leaves only pairs $(i,j) \in B \times V^c$ and $(i,j) \in V \times V$.  We also note that the events in \eqref{eq:stability-event-triple} are mutually independent conditioned on $\bx,\by,\bz$.  By Lemmas \ref{lem:no-big-y}, \ref{lem:two-point} and \ref{lem:integral-compute}, we may ignore the contribution from the set $\cB = \bigcup_i\{y_i > 2 \log^2 n / \sqrt{n} \vee z_i > 2 \log^2 n / \sqrt{n}\}$\,.  On $\cB^c$, for each pair $(i,j) \in (V \times V) \setminus \{\Pi \cup \Pi_1 \cup \Pi_2\}$ we have $$\P(\cS_{i,j}(\Pi_1) \cap \cS_{i,j}(\Pi_2) \,| \Pi, \bx,\by,\bz) = 1 \pm \widetilde{O}(n^{-1})\,.$$
    Since there are most $|V|^2 \leq n^{1/2}$ many such pairs, this shows \eqref{eq:double-disjoint-factor}\,.  Integrating using Lemma \ref{lem:integral-compute} as in the proof of Lemma \ref{lem:two-point-conditional} completes the proof. 
\end{proof}

\subsection{Overlapping cycles}\label{subsec:quasirandom}
We will prove the following stronger version of Lemma \ref{lem:no-overlaps-undirected}.  

\begin{lemma}\label{lem:no-overlaps-directed}  
    Let $\vec{\Pi}_1$ and $\vec{\Pi}_2$ be orientations of matchings $\Pi_1,\Pi_2$ so that $|\Pi_i \triangle \Pi| \leq n^{1/4}$ for $i \in \{1,2\}$.  Define $A_{1,2} = |\{j \in A_1 \cap A_2 : \Pi_1(j) \neq \Pi_2(j)\}|$.  Then \begin{equation*}
        \P(\vec{\Pi}_1, \vec{\Pi}_2 \,|\,\Pi) \leq (1 + \widetilde{O}(n^{-1/3})) 4^{|A_{1,2}|} n^{-|\Pi_1 \cup \Pi_2 \setminus \Pi|}\,.
    \end{equation*}
\end{lemma}

\begin{remark}\label{remark:k-point}
The proof of Lemma \ref{lem:no-overlaps-directed} can in fact be turned into an asymptotic \emph{equality}, although we do not require it for our work here.  Further, our proof indicates how to handle the asymptotics of an arbitrary $k$-tuple: if one has a tuple of oriented cycles $\vec{\Pi}_j$ for $j \in [k]$ for fixed $k$, then one may construct the directed graph $\bigcup_{j \in [k]} \vec{\Pi}_j \setminus \Pi$ by orienting $\vec{e} \in \vec{\Pi}_j \setminus \Pi$ as coming from $A_j$ and going to $B_j$.  Setting $d_{\mathrm{in}}(j)$ and $d_{\mathrm{out}}(j)$ to be the in and out degree of vertex $j$ in this graph, the proof of Lemma \ref{lem:no-overlaps-directed} shows that one expects \begin{equation*}
    \P(\vec{\Pi}_1,\ldots,\vec{\Pi}_k \,|\,\Pi) \sim \prod_{j \in [n]} \left(d_{\mathrm{in}}(j)!d_{\mathrm{out}}(j)!\right) n^{-|\bigcup_j \Pi_j \setminus \Pi|}
\end{equation*}
provided the orientations are compatible.
\end{remark}

We now quickly note how one can deduce Lemma \ref{lem:no-overlaps-undirected} before proving Lemma \ref{lem:no-overlaps-directed}. 

\begin{proof}[Proof of Lemma \ref{lem:no-overlaps-undirected}]
    There are $4$ possible choices for orientations for the pair $\Pi_1,\Pi_2$.  Noting $\kappa = 2|A_{1,2}|$ and applying Lemma \ref{lem:no-overlaps-directed} completes the proof. 
\end{proof}

\begin{proof}[Proof of Lemma \ref{lem:no-overlaps-directed}]
    Define $A_1^\circ = A_1 \setminus A_{1,2}$ and $A_2^\circ = A_2 \setminus (A_1 \cup A_{1,2})$. Define $B_{1,2}$, $B_1^\circ$ and $B_2^\circ$ analogously.  Define $V = A_1 \cup A_2 \cup B_1 \cup B_2$.  Let $\by$ and $\bz$ be defined by $y_i = X_{i,\Pi_1(i)}$ and $z_i = X_{i,\Pi_2(i)}$ as before.  We again contain the event of stability of $\vec{\Pi}_1,\vec{\Pi}_2$ within 
    \begin{equation*}
        \bigwedge_{(i,j) \in V \times V^c} \cS_{i,j}(\Pi_1) \cap \cS_{i,j}(\Pi_2)\,.
    \end{equation*}

    Let $\vec{E}$ denote the set on $\bx,\by,\bz$ corresponding to orientations $\vec{\Pi}_1, \vec{\Pi}_2$.  On the set $B_{1,2}$ we have $z_i \geq x_i$ and $y_i \geq x_i$.  There are $2^{|B_{1,2}|}$ many choices total for picking the inequalities  $z_i > y_i > x_i$ or $y_i > z_i > x_i$ for each $i \in B_{1,2}$; we will see that each gives the same contribution and they have disjoint supports, so let $\vec{C}$ denote the orientation having $y_i > z_i > x_i$ for all $i \in B_{1,2}$.  Then for $\bx,\by,\bz$ in $\vec{E}$ and $\vec{C}$ we have 
    \begin{align*}
        \P(\vec{\Pi}_1,\vec{\Pi}_2 \,|\,\Pi,\bx,\by,\bz) &\leq \prod_{i \in A_1} \one_{y_i \leq x_i}\prod_{i \in A_2} \one_{z_i \leq x_i} \prod_{(i,j) \in B_1^{\circ} \times V^c} \frac{1 - y_i x_j}{ 1 - x_i x_j} \prod_{(i,j) \in B_2^\circ \times V^c} \frac{1 - z_i x_j}{1 - x_i x_j} \\
        &\qquad \prod_{(i,j) \in B_{1,2}^\circ \times V^c} \frac{1 - y_i x_j}{1 - x_i x_j} \one_{z_i \in [x_i,y_i]}\,.
    \end{align*}
    We note that \begin{equation*}
        \int_{[0,1]^{A_1^\circ \cup A_{1,2}}} \int_{[0,1]^{A_2^\circ \cup A_{1,2}}}  \prod_{i \in A_1} \one_{y_i \leq x_i}\prod_{i \in A_2} \one_{z_i \leq x_i} \,d\by \,d\bz = \prod_{i \in A_1^\circ \cup A_2^\circ} x_i \prod_{i \in A_{1,2}} x_i^2\,.
    \end{equation*}
    Applying Lemma \ref{lem:integral-compute} and bounding as in Lemma \ref{lem:two-point-conditional} we have \begin{align*}
         \int_{[0,1]^{B_1^\circ \cup B_{1,2}}} &\int_{[0,1]^{B_2^\circ \cup B_{1,2}}} \prod_{(i,j) \in B_1^{\circ} \times V^c} \frac{1 - y_i x_j}{ 1 - x_i x_j} \prod_{(i,j) \in B_2^\circ \times V^c} \frac{1 - z_i x_j}{1 - x_i x_j} \prod_{(i,j) \in B_{1,2} \times V^c} \frac{1 - y_i x_j}{1 - x_i x_j} \one_{z_i \in [x_i,y_i]} \,d\by\,d\bz \\
         &\leq (1 + \widetilde{O}(n^{-1/3})) n^{-|B_1^\circ|/2 - |B_2^\circ|/2 - |B_{1,2}|} \,.
    \end{align*}
    We then see that \begin{align*}
        \P(\vec{\Pi}_1,\vec{\Pi}_2,\cG_\Pi, \vec{C} \,|\,\Pi) \P(\Pi) &\leq (1 + \widetilde{O}(n^{-1/3})) n^{-\frac{|B_1^\circ|}{2} - \frac{|B_2^\circ|}{2} - |B_{1,2}|}\E_\bx\left[\one_{\cG_\Pi}  \prod_{i \in A_1^\circ \cup A_2^\circ} x_i \prod_{i \in A_{1,2}} x_i^2 \P(\Pi\,|\,\bx) \right] \\
        &= (1 + \widetilde{O}(n^{-1/3})) n^{-|B_1^\circ| - |B_2^\circ| - 2|B_{1,2}|} 2^{|A_{1,2}|} \frac{e^{1/2}}{(n-1)!!}
    \end{align*}
    by Lemma \ref{lem:integral-compute}, where we used that $|B_1^\circ| = |A_1^\circ|$ and similarly for $|A_2^\circ|$ and $|A_{1,2}|$.   Noting that $|B_1^\circ| + |B_2^\circ| + 2|B_{1,2}^\circ| = |\Pi_1 \cup \Pi_2 \setminus \Pi|$ and summing over all $2^{|B_{1,2}|} = 2^{|A_{1,2}|}$ choices for the ordering $\vec{C}$ completes the proof. 
\end{proof}

\subsection{Combining disjoint stable matchings}\label{subsec:combine}
Finally, we show that if two stable matchings differ from $\Pi$ on disjoint sets, then the matching that simultaneously differs from $\Pi$ on both sets is also stable with high probability. The idea here is that if the matching is not stable, we must have an obstruction occur between the two sets where the matchings differ from $\Pi$. When the sets are not too large, the probability of such an obstruction occurring is small. 
\begin{proof}[Proof of Lemma~\ref{lem:not-union}]
    Write \begin{align}\label{eq:demorgan}
        \{\Pi, \Pi_1, \Pi_2, \neg \Pi_{1,2}\} = \bigwedge_{ij \notin \Pi} \cS_{i,j}(\Pi) \cap \bigwedge_{ij \notin \Pi_1}\cS_{i,j}(\Pi_1) \cap \bigwedge_{ij \notin \Pi_2} \cS_{i,j}(\Pi_2) \cap \left(\bigcup_{ij \notin \Pi_{1,2}} \cS_{i,j}(\Pi_{1,2})^c \right)\,.
    \end{align}
    Let $V_j$ be the nodes of $\Pi_j\triangle \Pi$ and $V = V_1 \cup V_2$.  Note that for $(i,j) \in V^c \times V^c$ we have $\cS_{i,j}(\Pi_{1,2}) = \cS_{i,j}(\Pi)$.  Similarly, for $(i, j) \in V_1 \times V^c$ we have $\cS_{i,j}(\Pi_{1,2}) = \cS_{i,j}(\Pi_1)$ and analogously for $(i,j) \in V_2 \times V^c$.  For $(i,j) \in V_1^2$ with $(i,j) \notin \Pi_1$ we have $\cS_{i,j}(\Pi_{1,2}) = \cS_{i,j}(\Pi_1)$ and analogously for $(i,j) \in V_2^2$.  This implies that the only event in the union in \eqref{eq:demorgan} that may hold simultaneously with the previous intersection is for a pair with $i \in V_1, j \in V_2$, i.e. \begin{align*}
    \{\Pi_1, \Pi_2, \Pi, \neg \Pi_{1,2}\} &\subset \bigcup_{i \in V_1, j \in V_2} \cS_{i,j}^c(\Pi_{1,2}) \cap \{\Pi,\Pi_1,\Pi_2\} \\
    &= \bigcup_{i \in V_1, j \in V_2} \{X_{i,j} < X_{i,\Pi_1(i)} \wedge X_{j,i} < X_{j,\Pi_2(j)}\} \cap \{\Pi,\Pi_1,\Pi_2\}\,.
    \end{align*}
    If we fix a pair, $a \in V_1, b \in V_2$ and orientations $\vec{\Pi}_1,\vec{\Pi}_2$, then in the notation of Lemma \ref{lem:disjoint-factor} we have  \begin{align*}
        \P(\vec{\Pi}_1,\vec{\Pi}_2, \,& X_{i,j} < y_a,  X_{j,i} < z_{b}  \,|\,\Pi, \bx,\by,\bz) \\
        &\leq y_a z_b \P\left(\bigwedge_{(i,j) \in V \times V^c} \cS_{i,j}(\Pi_1) \wedge \cS_{i,j}(\Pi_2) \,\big|\Pi,\bx,\by,\bz \right) \\
        &=y_az_b\prod_{i \in A_1} \one_{y_i < x_i} \prod_{i \in A_2} \one_{z_i < x_i} \prod_{(i,j) \in B_1 \times V^c} \frac{1 - y_i x_j }{1 - x_i x_j}  \prod_{(i,j) \in B_2 \times V^c} \frac{1 - z_i x_j }{1 - x_i x_j} \,.
    \end{align*} 
    Note that by Lemma \ref{lem:no-big-y} we may intersect with the event that $y_a z_b \leq n^{-1 + o(1)}$.  By integrating as in the proof of Lemma \ref{lem:disjoint-factor} we see 
    \begin{align*}
        \P(\Pi_1,\Pi_2, \cS_{a,b}(\Pi_{1,2})^c, \cG_\Pi \,|\,\Pi) \leq n^{-1 + o(1)} \P(\Pi_1, \cG_\Pi \,|\, \Pi) \P(\Pi_2, \cG_\Pi \,|\, \Pi). 
    \end{align*}
    Summing over the at most $n^{2\alpha}$ pairs with $a \in V_1, b \in V_2$ completes the proof.
\end{proof}

\section*{Acknowledgements}
BC is supported by NSF GRFP 2101064. MM is supported in part by NSF grants DMS-2336788 and DMS-2246624.  The authors thank Julian Sahasrabudhe for comments on a previous draft.

\bibliographystyle{plain}
\bibliography{main.bib}

\appendix

\section{Direct integral computations} \label{app:proof-G}

We start with a first simple bound on the integrand:
\begin{lemma}\label{lem:integrand-UB}
    For all $\bx \in [0,1]^n$ we have $$\prod_{ij \notin \Pi}(1 - x_i x_j) \leq e^{81/16}\exp\left(-\frac{1}{2}\left(\sum_i x_i\right)^2  \right)\,.$$
\end{lemma}
\begin{proof}
    Using the inequality $1 - x \leq e^{-x - x^2/2}$ for $x \geq 0$ we see \begin{align*}
        \prod_{ij \notin \Pi}(1 - x_i x_j) &\leq \exp\left(- \sum_{ij \notin \Pi} \left(x_ix_j + \frac{x_i^2 x_j^2}{2}\right)\right)\,.
    \end{align*}
    Note that \begin{equation*}
        \sum_{ij \notin \Pi} x_i x_j = \frac{1}{2}\left(\sum_{i} x_i\right)^2 - \frac{1}{2}\sum_{i} x_i^2 - \sum_{ij \in \Pi} x_i x_j \geq \frac{1}{2} \left(\sum_{i} x_i \right)^2 - \frac{3}{2} \sum_{i} x_i^2\,.
    \end{equation*}
    The same shows \begin{equation*}
        \sum_{ij \notin \Pi} x_i^2 x_j^2 \geq \frac{1}{2}\left(\sum_i x_i^2 \right)^2 - \frac{3}{2}\sum_i x_i^4 \geq \frac{1}{2}\left(\sum_i x_i^2 \right)^2 - \frac{3}{2}\sum_i x_i^2 
    \end{equation*}
    where we used $x_i \in [0,1]$ to bound $x_i^4 \leq x_i^2$.      This shows 
    \begin{align*}
        \prod_{ij \notin \Pi}(1 - x_i x_j) \leq \exp\left( - \frac{1}{2}\left(\sum_{i} x_i\right)^2 + \frac{9}{4} \sum_i x_i^2 - \frac{1}{4} \left(\sum_i x_i^2 \right)^2\right) \leq \exp\left( - \frac{1}{2}\left(\sum_{i} x_i\right)^2 + \frac{81}{16}\right)
    \end{align*}
    where we used that $\frac{9}{4}t - \frac{1}{4} t^2 \leq \frac{81}{16}$ for all $t \geq 0$.
\end{proof}

We will need two classic concentration estimates. The first is Bernstein's inequality (see, e.g., \cite[Theorem 2.9.1]{vershynin2018high}): 
\begin{fact}\label{fact:bernstein}
    Let $Y_j$ be i.i.d.\ random variables with $\E Y = 0$ so that $\P(|Y| \geq y) \leq 2e^{-q y}$ for some $q$.  Then \begin{equation*}
        \P\left(\left|\sum_{i = 1}^n Y_i \right| \geq t \right) \leq \exp\left(- c_q \min\left\{\frac{t^2}{n},t\right\} \right)\,.
    \end{equation*}  
\end{fact}

We also need a variant for random variables with stretched exponential tails (see, e.g., \cite{brosset2022large}).
\begin{fact}\label{fact:stretched-exponential}
    Let $Y_j$ be i.i.d.\ random variables with $\E Y = 0$ so that $\log\P(|Y| \geq y) \sim -q y^{1/2}$ for some $q$ as $y \to \infty$.  Then for each fixed $\alpha \geq 1/2$ we have  \begin{equation*}
        \P\left(\left|\sum_{i= 1}^n Y_i\right| \geq n^{\alpha} \right) \leq \exp\left( - \min\{n^{2\alpha - 1 + o(1)},n^{\alpha/2 + o(1)} \}\right)\,.
    \end{equation*}
\end{fact}

This then applies to both $\sum_i (X_i^2 - 2)$ and $\sum_{ij \in \Pi} (X_i X_j - 1)$ for $X_i$ being i.i.d.\ standard exponential random variables. The proof of Lemma \ref{lem:prob-G-UB} will now follow easily.

\begin{proof}[Proof of Lemma~\ref{lem:prob-G-UB}]
    Combining Lemmas \ref{lem:single:exact} and \ref{lem:integrand-UB} that $$\P(\Pi \,|\, X_{i,\Pi(i)} = x_i) = \prod_{ij \notin \Pi} (1 - x_ix_j) \leq e^{6} \exp\left(- \frac{1}{2} \left(\sum_{i} x_i \right)^2 \right)\,.$$
    Set $s = \sum_i x_i$.  We first show that the contribution to $\P(\Pi)$ on the event $\cB = \{|s - \sqrt{n}| \geq 10 \log n\}$ is small.  By changing variables to the simplex we see  
    \begin{align*}
        \P(\Pi \cap \cB) \leq e^{6} \int_{\bx \in [0,1]^n \cap \cB} e^{-s^2 / 2} \,d\bx \leq e^{6} \int_{|s - \sqrt{n}| \geq 10 \log n} e^{-s^2 / 2} \frac{s^{n-1}}{(n-1)!}\,ds\,. 
    \end{align*}
    Note that $-s^2 / 2 + (n-1) \log s$ is concave with maximum at $s_\ast = \sqrt{n-1}$. We also note that the integral over all $s$ is given by \begin{equation*}
        \int_0^\infty e^{-s^2 / 2} s^{n-1 } \,ds = (n-2)!! = n^{o(1)} \left(\frac{n-1}{e}\right)^{(n-1)/2}\,.
    \end{equation*}
    By monotonicity of the integrand we may bound \begin{align*}
        \int_0^{\sqrt{n} - 10 \log n} e^{-s^2/2} s^{n-1} \,ds &\leq \sqrt{n} \exp\left(- \frac{(\sqrt{n} - 10 \log n)^2}{2} + (n-1) \log (\sqrt{n} - 10 \log n) \right) \\
        &=  \left(\frac{n-1}{e} \right)^{(n-1)/2} e^{-(100 + o(1)) \log^2 n}\,.
    \end{align*}
    For the upper tail we parameterize $s = \sqrt{n-1}(1 +t)$; the integrand may be bound via \begin{equation*}
        e^{-s^2 / 2}s^{n-1} = \left(\frac{n-1}{e} \right)^{(n-1)/2} e^{- \left(\frac{n-1}{2}\right) \left(2t + t^2 - 2\log(1 + t) \right)} \leq \left(\frac{n-1}{e} \right)^{(n-1)/2} e^{- t^2\left(\frac{n-1}{2}\right)}
    \end{equation*}
    where in the last bound we used $t \geq \log(1 + t)$ for $t > -1$.  Integrating shows that $$\P(\Pi \cap \cB) \leq e^{-(100 + o(1)) \log^2n} \frac{(n-2)!!}{(n-1)!}= \frac{e^{-(100 + o(1)) \log^2 n}}{(n-1)!!}\,.$$
    For the other events in $\cG_\Pi$ we recall that if $(x_1,\ldots,x_n) \in [0,\infty)^n$ is chosen uniformly at random conditioned on $\sum x_j = s$ then we have $x_j = s \frac{Y_j}{\sum_i Y_i} =: s \widehat{Y}_j$ where $Y_j$ are i.i.d.\ standard exponential random variables. We then bound
    \begin{align*}
        \P\left(X_{i,\Pi(i)} \geq \frac{\log^2 n}{\sqrt{n}} , \cB^c, \Pi\right) &\leq e^6 \int_{0}^{\sqrt{n} + 10 \log n} \frac{e^{-s^2/2} s^{n-1}}{(n-1)!} \cdot \P\left(\cG'^c\right)\,ds\,.
    \end{align*}
    where $\cG'$ is defined by \begin{align}
        \cG' = \left\{\forall~i : \widehat{Y}_i \leq \frac{\log^2 n}{s\sqrt{n}}\right\} \cap \left\{ \left|\sum_{(i,j) \in \Pi} \widehat{Y}_i\widehat{Y}_j  - \frac{1}{2s^2} \right| \leq \frac{1}{s^2 n^{1/3}} \right\} \cap \left\{\left|\sum_{i} \widehat{Y}_i^2s^2 - 2\right| \leq n^{-1/3} \right\} \label{eq:cG'-def}\,.
    \end{align}
    By Facts \ref{fact:bernstein} and  \ref{fact:stretched-exponential} we have $\P(\cG'^c) \leq e^{-\Omega(\log^2 n)}$.  Using the fact that $$\int_0^{\infty} s^{n-1} e^{-s^2 / 2}\,ds = (n-2)!!$$ completes the proof. 
\end{proof}

\begin{proof}[Proof of Lemma \ref{lem:integral-compute}]
    First note that on the event $\cG_\Pi$ using the notation $s = \sum_i x_i$ we have 
    \begin{align*}
        \log\left(\prod_{(i,j) \notin \Pi} (1 - x_ix_j)\right) &= -\sum_{(i,j) \notin \Pi} \left(x_{i}x_j + \frac{x_i^2 x_j^2}{2}\right) \pm \widetilde{O}(n^{-1}) \\
        &= -\frac{s^2}{2} + \frac{1}{2}\sum_i x_i^2 + \sum_{(i,j) \in \Pi} x_i x_j - \frac{1}{4} \left(\sum_{i} x_i^2\right)^2 \pm \widetilde{O}(n^{-1}) \\
        &= -\frac{s^2}{2} + \frac{1}{2} \pm {O}(n^{-1/3}) \,.
    \end{align*}
    We then simplify \begin{align*}
        \int_{\bx \in \cG_\Pi} &\prod_{j \in B_1} x_j \prod_{j \in B_2} x_j^2 \prod_{(i,j) \notin \Pi} (1 - x_ix_j) \,d\bx = (e^{1/2} \pm O(n^{-1/3})) \int_{\bx \in \cG_\Pi} \prod_{j \in B_1} x_j \prod_{j \in B_2} x_j^2  \cdot e^{-s^2/2} \,d\bx\,. \end{align*}
        Letting $Y_j$ be i.i.d.\ standard exponential variables, and $\widehat{Y}_j = \frac{Y_j}{\sum_{i = 1}^n Y_i}$ we may change variables to the simplex and see
    \begin{align*}
        \int_{\bx \in \cG_\Pi} \prod_{j \in B_1} x_j \prod_{j \in B_2} x_j^2  \cdot e^{-s^2/2} \,d\bx = \E\left[ \prod_{j \in B_1} \widehat{Y}_j \prod_{j \in B_2} \widehat{Y}_j^2 \one\{\cG'\}\right] \int_{\sqrt{n} - 10 \log n}^{\sqrt{n} + 10\log n} \frac{s^{n+|B_1| + 2|B_2|-1}}{(n-1)!} e^{-s^2/2} \,ds
    \end{align*}
    where $\cG'$ is defined in \eqref{eq:cG'-def}.   
    \begin{claim}\label{cl:expectation-Y's}
    We have
        $$\E\left[ \prod_{j \in B_1} \widehat{Y}_j \prod_{j \in B_2} \widehat{Y}_j^2 \one\{\cG'\}\right] = (1 \pm O(n^{-1/2})) n^{-|B_1| - 2|B_2|} 2^{|B_2|}\,.$$
    \end{claim}
    \begin{proof}[Proof of Claim \ref{cl:expectation-Y's}]
        Set $\cG_1',\cG_2'$ and $\cG_3'$ to be the three events in \eqref{eq:cG'-def}.  Set $Z = \sum_{i \notin B_1 \cup B_2} Y_i$ and set  $\cG_4' = \{Z \leq n/2\}$ and note that by Bernstein's inequality we have $\P(\cG_4'^c) \leq \exp(- \Omega(n))$ and so we may condition on $\cG_4'$.  On $\cG_4'$ and $\cG_1'$ we note that \begin{align*}
            \prod_{j \in B_1} \widehat{Y}_j \prod_{j \in B_2} \widehat{Y}_j^2 &= (1 \pm \widetilde{O}(n^{-3/4}))^{|B_1| + 2|B_2|} Z^{-|B_1| - 2|B_2|}\prod_{j \in B_1} {Y}_j \prod_{j \in B_2} {Y}_j^2  \\
            &= (1 \pm \widetilde{O}(n^{-1/2})) Z^{-|B_1| - 2 |B_2|}\prod_{j \in B_1} {Y}_j \prod_{j \in B_2} {Y}_j^2 \,.
        \end{align*}
        By Fact \ref{fact:bernstein} we have $\P(|\sum_i Y_i - n| \geq n^{2/3}) = \exp(-n^{-1/3 + o(1)})$ and so by Fact \ref{fact:stretched-exponential} we have that $$\P((\cG_2')^c \cup (\cG_3')^c) \leq e^{-n^{1/3 + o(1)}}\,.$$
    We may thus remove the conditioning on $\cG_2',\cG_3'$.  We then have \begin{align*}
        \E\left[ \prod_{j \in B_1} \widehat{Y}_j \prod_{j \in B_2} \widehat{Y}_j^2 \one\{\cG_1' \cap \cG_4'\}\right] &= (1 \pm \widetilde{O}(n^{-1/2})) \E\left[ Z^{-|B_1| - 2|B_2|} \prod_{j \in B_1} {Y}_j \prod_{j \in B_2} {Y}_j^2  \one\{\cG_1' \cap \cG_4'\} \right]\,.
    \end{align*}
    For an upper bound on this integral, we may omit the indicator variable and precisely compute $$\E\left[ Z^{-|B_1| - 2|B_2|} \prod_{j \in B_1} {Y}_j \prod_{j \in B_2} {Y}_j^2  \right] = (1 \pm O(n^{-1/2}))n^{-|B_1| - 2 |B_2|} 2^{|B_2|}\,.$$
    For a matching lower bound, we may first sample $Z$ and note that by a Bernstein's inequality we have that $Z  = n \pm \widetilde{O}(n^{-1/2})$ with probability $\geq 1 - 1/n$, and so we may lower bound the expectation over $Z$ by this quantity.  Taking the expectation over the variables $Y_j$ for $j \in B_1 \cup B_2$ completes the claim.
    \end{proof}
    To complete the proof, note \begin{align*}\int_{\sqrt{n} - 10 \log n}^{\sqrt{n} + 10 \log n} s^{n+|B_1| + 2|B_2|-1} e^{-s^2/2} \,ds &=(1 \pm O(n^{-1})) \int_{0}^{\infty} s^{n+|B_1| + 2|B_2|-1} e^{-s^2/2} \,ds  \\
     &=(1 \pm O(n^{-1}))(n+|B_1| + 2|B_2| - 2)!! \,. \end{align*} 
     Noting that \begin{align*}
         \frac{(n+|B_1| + 2|B_2| - 2)!!}{(n-1)!} n^{-|B_1| - 2 |B_2|} &= (1 + O(n^{-1/2})) \frac{(n-2)!!}{(n-1)!} n^{-|B_1|/2 - |B_2|} \\
         &=  (1 + O(n^{-1/2})) \frac{1}{(n-1)!!} n^{-|B_1|/2 - |B_2|}
     \end{align*}
     completes the proof.
\end{proof}

\end{document}